\documentclass[a4paper,11pt]{article}
\usepackage{amsmath, amsfonts, amscd, amssymb, amsthm, enumerate, xypic}

\def\bfB{\mathbf{B}}

\def\ad{\text{ad}}
\def\diff{\text{d}}

\newcommand{\modu}{\operatorname{mod}}
\newcommand{\WS}{\mathcal{N}\mathcal{S}}
\newcommand{\WA}{\mathcal{N}\mathcal{A}}
\newcommand{\WH}{\mathcal{N}\mathcal{H}}
\newcommand{\Pgros}{\mathbb{P}}
\renewcommand{\epsilon}{\varepsilon}

\newcommand{\Mat}{\operatorname{M}}
\newcommand{\Mats}{\operatorname{S}}
\newcommand{\Mata}{\operatorname{A}}
\newcommand{\Math}{\operatorname{H}}

\newcommand{\GL}{\operatorname{GL}}
\newcommand{\Ker}{\operatorname{Ker}}

\newcommand{\End}{\operatorname{End}}
\newcommand{\NT}{\operatorname{NT}}
\newcommand{\Vect}{\operatorname{span}}
\newcommand{\im}{\operatorname{Im}}
\newcommand{\Rad}{\operatorname{Rad}}
\newcommand{\tr}{\operatorname{tr}}
\newcommand{\Tr}{\operatorname{Tr}}

\newcommand{\rk}{\operatorname{rk}}

\renewcommand{\setminus}{\smallsetminus}

\def\F{\mathbb{F}}
\def\K{\mathbb{K}}

\def\C{\mathbb{C}}

\def\calA{\mathcal{A}}

\def\calF{\mathcal{F}}

\def\calH{\mathcal{H}}

\def\calL{\mathcal{L}}

\def\calS{\mathcal{S}}

\def\calU{\mathcal{U}}
\def\calV{\mathcal{V}}

\def\lcro{\mathopen{[\![}}
\def\rcro{\mathclose{]\!]}}

\theoremstyle{definition}
\newtheorem{Def}{Definition}[section]
\newtheorem{Not}[Def]{Notation}

\theoremstyle{plain}
\newtheorem{theo}{Theorem}[section]
\newtheorem{prop}[theo]{Proposition}

\newtheorem{lemma}[theo]{Lemma}
\newtheorem{claim}{Claim}

\theoremstyle{plain}
\newtheorem{conj}{Conjecture}

\theoremstyle{remark}

\title{The structured Gerstenhaber problem (I)}
\author{Cl\'ement de Seguins Pazzis\footnote{Universit\'e de Versailles Saint-Quentin-en-Yvelines, Laboratoire de Math\'ematiques
de Versailles, 45 avenue des Etats-Unis, 78035 Versailles cedex, France, dsp.prof@gmail.com}}

\begin{document}

\thispagestyle{plain}

\maketitle

\begin{abstract}
Let $b$ be a symmetric or alternating bilinear form on a finite-dimensional vector space $V$.
When the characteristic of the underlying field is not $2$, we determine the greatest dimension for a linear subspace of nilpotent $b$-symmetric or $b$-alternating endomorphisms of $V$, expressing it as a function of the dimension, the rank, and the Witt index of $b$. Similar results are obtained for subspaces of nilpotent $b$-Hermitian endomorphisms when $b$ is a Hermitian form with respect to a non-identity involution.
In three situations ($b$-symmetric endomorphisms when $b$ is symmetric, $b$-alternating endomorphisms when $b$ is alternating, and
$b$-Hermitian endomorphisms when $b$ is Hermitian and the underlying field has more than $2$ elements), we also characterize the linear subspaces with the maximal dimension.

Our results are wide generalizations of results of Meshulam and Radwan~\cite{MeshulamRadwan}, who tackled the case of a non-degenerate symmetric bilinear form over the field of complex numbers, and recent results of Bukov\v{s}ek and Omladi\v{c}~\cite{BukovsekOmladic}, in which the spaces with maximal dimension were determined when the underlying field is the one of complex numbers, the bilinear form $b$ is symmetric and non-degenerate, and one considers
$b$-symmetric endomorphisms.
\end{abstract}




\section{Introduction}

\subsection{The problem}\label{ProblemSection}

Throughout, $\F$ denotes an arbitrary field and we consider a
finite-dimensional vector space $V$ over $\F$, equipped with a bilinear form $b : V \times V \rightarrow \F$
which is either symmetric ($\forall (x,y)\in V^2, \; b(x,y)=b(y,x)$)
or alternating ($\forall x \in V, \; b(x,x)=0$).
Given a subset $X$ of $V$, we denote by $X^\bot$ the set of all vectors $y \in V$ such that $\forall x \in X, \; b(x,y)=0$: this is a linear subspace of $V$. We say that $b$ is \textbf{non-degenerate} whenever $V^\bot=\{0\}$. We recall that if $b$ is non-degenerate and $X$ is a linear subspace of $V$, we have $\dim X+\dim X^\bot=\dim V$ and $(X^\bot)^\bot=X$.
A linear subspace $X$ of $V$ is called \textbf{totally singular} for $b$ if $X \subset X^\bot$;
the maximal dimension for such a subspace is called the \textbf{Witt index} of $b$.
Finally, we say that $b$ is \textbf{non-isotropic} whenever its Witt index equals zero, or equivalently
$b(x,x)\neq 0$ for all $x \in V \setminus \{0\}$.

An endomorphism $u$ of $V$ is called \textbf{$b$-symmetric} (respectively, \textbf{$b$-alternating})
whenever the bilinear form
$$(x,y) \in V^2 \mapsto b(x,u(y))$$
is symmetric (respectively, alternating).
The set of all $b$-symmetric endomorphisms is denoted by $\calS_b$. The set of all $b$-alternating
ones is denoted by $\calA_b$. Both sets are linear subspaces of the
space $\End(V)$ of all endomorphisms of $V$. At this point, it should be noted that our terminology can diverge
from more traditional ones if $b$ is alternating: set $\epsilon:=1$ if $b$ is symmetric, otherwise set $\epsilon:=-1$.
An endomorphism $u \in \End(V)$ is $b$-symmetric if and only if
$$\forall (x,y) \in V^2, \; b(u(x),y)=\epsilon\,b(x,u(y)),$$
and some mathematicians would rather call such an endomorphism $b$-\emph{skew-symmetric} if $\epsilon=-1$.
However, the consistency of the results from this article will amply justify our terminology.
In any case, if $u \in \End(V)$ is $b$-alternating then
$$\forall (x,y) \in V^2, \; b(u(x),y)=-\epsilon\,b(x,u(y)),$$
and the converse holds if the characteristic of $\F$ is not $2$. Hence, whenever $u$ is $b$-symmetric of $b$-alternating, there exists an $\epsilon' \in \{-1,1\}$ such that
$$\forall (x,y) \in V^2, \; b(u(x),y)=\epsilon'\,b(x,u(y)).$$

Finally, a subset of an $\F$-algebra is called \textbf{nilpotent} when all its elements are nilpotent.

\vskip 5mm
One of the most celebrated results in modern linear is Gerstenhaber's theorem \cite{Gerstenhaber} on linear subspaces of nilpotent matrices, which reads as follows:

\begin{theo}[Gerstenhaber (1958)]
Let $\calV$ be a nilpotent linear subspace of the space $\Mat_n(\F)$ of all $n$-by-$n$ square matrices with entries in $\F$.
Then,
$$\dim \calV \leq \dbinom{n}{2}$$
and equality holds if and only if there exists an invertible matrix $P \in \GL_n(\F)$ such that $P \calV P^{-1}$
equals the space $\NT_n(\F)$ of all strictly upper-triangular matrices of $\Mat_n(\F)$.
\end{theo}

In the original result of Gerstenhaber, the field was required to have at least $n$ elements.
This provision was later lifted \cite{Serezhkin} and the result was even generalized to skew fields
\cite{dSPGerstenhaberskew}.
Recent advances have been made in this problem: the inequality statement has been generalized to so-called \emph{trivial spectrum}
subspaces \cite{Quinlan,dSPlargerank}, i.e.\ linear subspaces in which the only possible eigenvalue of a matrix in the underlying field is zero; moreover the classification of trivial spectrum subspaces with the maximal dimension has been reduced to the classification of non-isotropic bilinear forms up to similarity \cite{dSPaffinenonsingular} if the underlying field has more than $2$ elements.
The Gerstenhaber problem has also been tackled in the more general setting of a Lie algebra:
in \cite{MeshulamRadwan}, Meshulam and Radwan showed that, given a complex semisimple Lie algebra $\mathfrak{g}$,
a linear subspace of $\mathfrak{g}$ that consists only of ad-nilpotent elements has dimension at most
$\frac{1}{2}(\dim \mathfrak{g}-\rk \mathfrak{g})$. In \cite{DraismaKraftKuttler},
this result was generalized to an arbitrary algebraically closed field
(ad-nilpotency should then be replaced with nilpotency, and $\mathfrak{g}$ should be the Lie algebra of a reductive algebraic group $G$), and,
except in special cases over fields with characteristic $2$ or $3$, the spaces with maximal dimension have been shown to be the Lie algebras of the Borel subgroups of $G$.

Gerstenhaber's theorem can be translated into a statement on linear subspaces of endomorphisms:
recall that a \textbf{flag} of $V$ is an increasing sequence $\calF=(F_i)_{0 \leq i \leq p}$
of linear subspaces of $V$, and such a flag is \textbf{complete} if $p=\dim V$ (so that $\dim F_i=i$ for all $i \in \lcro 0,p\rcro$).
An endomorphism $u$ of $V$ is said to \textbf{stabilize} the flag $\calF$ whenever $u(F_i) \subset F_i$ for all $i \in \lcro 0,p\rcro$.
If in addition $u$ is nilpotent and $\calF$ is complete, this condition is equivalent to having $u(F_i) \subset F_{i-1}$ for all $i \in \lcro 1,p\rcro$.
The set of all nilpotent endomorphisms that stabilize the complete flag $\calF$ is then a linear subspace of $\End(V)$, and if
we choose an $\calF$-adapted basis $\bfB$ of $V$ -- i.e.\ a basis
$(e_1,\dots,e_p)$ of $V$ such that $F_i=\Vect(e_1,\dots,e_i)$ for all $i \in \lcro 0,p\rcro$ --
it is seen to be isomorphic to $\NT_p(\F)$ under the isomorphism $u \in \End(V) \mapsto \Mat_\bfB(u) \in \Mat_p(\F)$.
Hence, Gerstenhaber's theorem has the following equivalent formulation in terms of endomorphisms:

\begin{theo}
Let $\calV$ be a nilpotent linear subspace of $\End(V)$, and set $n:=\dim V$.
Then,
$$\dim \calV \leq \dbinom{n}{2}$$
and equality holds if and only if $\calV$ is the set of all nilpotent endomorphisms of $V$ that stabilize
some fixed complete flag of $V$.
\end{theo}

Instead of the above problem, which we call the \emph{standard} Gerstenhaber problem,
here we consider the \emph{structured} Gerstenhaber problem,
in which the nilpotent subspace $\calV$ under consideration is assumed to be either a subspace of $\calS_b$ or one of $\calA_b$
(note that this condition is trivial if $b=0$, in which case there is no difference with the standard Gerstenhaber problem).
In the structured Gerstenhaber problem, one asks the following questions, which are the counterpart of those in the standard Gerstenhaber problem:
\begin{itemize}
\item What is the maximal dimension for a nilpotent linear subspace of $\calS_b$ (respectively, of $\calA_b$)?
\item What are the nilpotent linear subspaces of $\calS_b$ (respectively, of $\calA_b$) with maximal dimension?
\end{itemize}
In the structured Gerstenhaber problem,  there are additional constraints and one expects the dimension bound to be lower than the one in the standard Gerstenhaber problem
(and of course the spaces with maximal dimension should have a different structure).
Similar questions can be raised with respect to Hermitian forms: so as not to burden this introduction with
considerations on Hermitian forms, we will say nothing about them here, and will tackle the problem in Section \ref{HermitianSection} only.

\vskip 3mm
One important remark is that the structured Gerstenhaber problem can be essentially reduced to the following separate problems:
\begin{itemize}
\item The structured Gerstenhaber problem in the special case when the form $b$ under consideration is non-degenerate.
\item The standard Gerstenhaber problem.
\end{itemize}
To understand this, recall that the \textbf{radical} of $b$ is defined as the linear subspace
$$\Rad(b):=V^\bot.$$
For $x \in V$, denote by $\overline{x}$ its class modulo $\Rad(b)$.
Then, $b$ induces a non-degenerate bilinear form $\overline{b}$ on $V/\Rad(b)$ such that
$\overline{b}(\overline{x},\overline{y})=b(x,y)$ for all $(x,y)\in V^2$. Moreover,
$\overline{b}$ is symmetric (respectively, alternating) if $b$ is symmetric (respectively, alternating).

We have the following elementary result:

\begin{lemma}
Let $u \in \End(V)$. Then, $u$ is $b$-symmetric (respectively, $b$-alternating) if and only if $u$ stabilizes
$\Rad(b)$ and the induced endomorphism $\overline{u}$ on $V/\Rad(b)$ is $\overline{b}$-symmetric (respectively,
$\overline{b}$-alternating).
\end{lemma}

\begin{proof}
Assume that $u$ is $b$-symmetric (respectively, $b$-alternating).
Then, for all $x \in \Rad(b)$, we have $\forall y \in V, \; b(u(x),y)=\epsilon\, b(x,u(y))=0$ for some $\epsilon \in \{1,-1\}$, whence $u(x) \in \Rad(V)$.
It is then easily checked that the induced endomorphism $\overline{u}$ on $V/\Rad(b)$ is $\overline{b}$-symmetric
(respectively, $\overline{b}$-alternating).
Conversely, assume that $u$ stabilizes $\Rad(b)$ and that the induced endomorphism $\overline{u}$ on $V/\Rad(b)$ is $\overline{b}$-symmetric.
For all $(x,y)\in V^2$, we have
$$b(x,u(y))=\overline{b}(\overline{x},\overline{u(y)})=\overline{b}(\overline{x},\overline{u}(\overline{y}))$$
and since we know that $(x',y')\in (V/\Rad(b))^2 \mapsto \overline{b}(x',\overline{u}(y'))$ is symmetric we conclude that so is
$(x,y) \in V^2 \mapsto b(x,u(y))$, i.e.\ $u$ is $b$-symmetric.
Likewise if $u$ stabilizes $\Rad(b)$ and the induced endomorphism $\overline{u}$ on $V/\Rad(b)$ is $\overline{b}$-alternating,
we see that $b(u,u(x))=\overline{b}(\overline{x},\overline{u}(\overline{x}))=0$ for all $x \in V$, whence $u$ is $b$-alternating.
\end{proof}

Hence, we obtain surjective linear mappings
$$R_s : u \in \calS_b \mapsto (u_{|\Rad(b)}, \overline{u})\in \End\bigl(\Rad(b)\bigr) \times \calS_{\overline{b}}.$$
and
$$R_a : u \in \calA_b \mapsto (u_{|\Rad(b)}, \overline{u})\in \End\bigl(\Rad(b)\bigr) \times \calA_{\overline{b}}.$$
It is easily seen that $R_s$ and $R_a$ have the same kernel, which equals the space of all the endomorphisms of $V$ that vanish everywhere on
$\Rad(b)$ and have their range included in $\Rad(b)$, and this vector space is naturally isomorphic to
$\calL(V/\Rad(b),\Rad(b))$. Moreover, an endomorphism $u \in \calS_b \cup \calA_b$ is nilpotent if and only if $u_{|\Rad(b)}$ and $\overline{u}$
are nilpotent. Hence, a nilpotent linear subspace of $\calS_b$ (respectively, of $\calA_b$) is just a linear subspace of the
inverse image under $R_s$ (respectively, under $R_a$) of $\calV_1 \times \calV_2$ for some nilpotent linear subspace $\calV_1$ of $\Rad(b)$
and some nilpotent linear subspace $\calV_2$ of $\calS_{\overline{b}}$ (respectively, of $\calA_{\overline{b}}$).

\subsection{Examples: the matrix viewpoint}\label{matrixExamples}

It is high time we gave examples of large spaces of nilpotent $b$-symmetric or $b$-alternating endomorphisms.
So as to better visualize things, we will start from the matrix viewpoint; the geometric
viewpoint will be dealt with in the next section.

First, some additional notation on matrices. For a non-negative integer $n$, remember that $\Mat_n(\F)$ denotes the algebra of all $n$-by-$n$ matrices with entries in $\F$ (we respectively denote by $I_n$ and $0_n$ its identity matrix and its zero matrix).
We denote by $\Mat_{n,p}(\F)$ the vector space of all $n$-by-$p$ matrices with entries in $\F$.
We denote by $\Mats_n(\F)$ the space of all symmetric $n$-by-$n$ matrices, and by $\Mata_n(\F)$ the space of all alternating $n$-by-$n$ matrices (a matrix $A \in \Mat_n(\F)$ is called alternating when the bilinear
form $(X,Y) \in (\F^n)^2 \mapsto X^T AY$ is alternating, which means that $A^T=-A$ and that all the diagonal entries of $A$ are zero).

Let $S \in \Mat_n(\F)$ be symmetric or alternating. We consider the corresponding
bilinear form $B : (X,Y) \in (\F^n)^2 \mapsto X^TSY$, which is symmetric if $S$ is symmetric, and alternating if $S$ is alternating. Then, given $M \in \Mat_n(\F)$, the endomorphism
$X \mapsto MX$ of $\F^n$ is $B$-symmetric (respectively, $B$-alternating) if and only if
$SM$ is symmetric (respectively, alternating). This motivates that we set
$$\calS_S:=\bigl\{M \in \Mat_n(\F) : SM \in \Mats_n(\F)\bigr\},$$
the set of all \textbf{$S$-symmetric} matrices, and
$$\calA_S:=\bigl\{M \in \Mat_n(\F) : SM \in \Mata_n(\F)\bigr\},$$
the set of all \textbf{$S$-alternating} matrices.
Now, say that we have an arbitrary basis $\bfB:=(e_1,\dots,e_n)$ of $V$ and
we take $$S:=\Mat_{\bfB}(b)=\bigl(b(e_i,e_j)\bigr)_{1 \leq i,j \leq n,}$$
which is
symmetric (respectively, alternating) if and only if $b$ is symmetric (respectively, alternating).
Then, the isomorphism of algebras
$$u \in \End(V) \mapsto \Mat_\bfB(u) \in \Mat_n(\F)$$
induces isomorphisms of vector spaces
$$u \in \calS_b \mapsto \Mat_\bfB(u) \in \calS_S \qquad \text{and} \qquad u \in \calA_b \mapsto \Mat_\bfB(u) \in \calA_S$$
that preserve nilpotency.

Assume now that $b$ is non-degenerate and alternating. Then $n$ is even,
the Witt index of $b$ is $\nu:=\frac{n}{2}$ and there is a basis
$\bfB$ of $V$ in which
$$\Mat_\bfB(b)=K_n:=\begin{bmatrix}
0_\nu & I_\nu \\
-I_\nu & 0_\nu
\end{bmatrix}.$$
A straightforward computation then shows that
$$\calS_{K_n}=\Biggl\{
\begin{bmatrix}
A & C \\
B & -A^T
\end{bmatrix} \mid A \in \Mat_\nu(\F), \; (B,C)\in \Mats_\nu(\F)^2
\Biggr\}.$$
Likewise,
$$\calA_{K_n}=\Biggl\{
\begin{bmatrix}
A & C \\
B & A^T
\end{bmatrix} \mid A \in \Mat_\nu(\F), \; (B,C)\in \Mata_\nu(\F)^2
\Biggr\}.$$
It is then easily checked that the space
$$\WS_\nu:=\biggl\{\begin{bmatrix}
N & C \\
0_\nu & -N^T
\end{bmatrix} \mid N \in \NT_\nu(\F), \; C \in \Mats_\nu(\F)
\biggr\}$$
consists of $K_n$-symmetric nilpotent matrices. Clearly this space has dimension
$\dbinom{\nu}{2}+\dbinom{\nu+1}{2}=\nu^2=\nu(n-\nu)$.
Next, the space
$$\WA_\nu:=\biggl\{\begin{bmatrix}
N & C \\
0 & N^T
\end{bmatrix} \mid N \in \NT_\nu(\F), \; C \in \Mata_\nu(\F)
\biggr\}$$
consists of $K_n$-alternating nilpotent matrices, and
one checks that it has dimension $2\dbinom{\nu}{2}=\nu(\nu-1)=\nu(n-\nu-1)$.

\vskip 3mm
Next, we consider the case when $b$ is non-degenerate and symmetric, and $\F$ does not have
characteristic $2$. Note that $b$ is equivalent to the orthogonal direct sum of a
$2\nu$-dimensional hyperbolic symmetric bilinear form and of a non-isotropic symmetric bilinear form.
Set $p:=n-2\nu$. Then, there is a basis $\bfB$ of $V$ together with a nonisotropic symmetric matrix
$P \in \Mats_p(\F)$ such that
$$\Mat_\bfB(b)=S_{\nu,P}:=\begin{bmatrix}
0_\nu & [0]_{\nu \times p} & I_\nu \\
[0]_{p \times \nu} & P & [0]_{p \times \nu} \\
I_\nu & [0]_{\nu \times p} & 0_\nu
\end{bmatrix}.$$
A straightforward computation then shows that the $b$-symmetric endomorphisms of $V$ are those that are represented in $\bfB$ by a matrix of the form
$$\begin{bmatrix}
A & (PC)^T & B \\
D & P^{-1} S & C \\
E & (PD)^T & A^T
\end{bmatrix}$$
in which $A \in \Mat_\nu(\F)$, $B,E$ belong to $\Mats_\nu(\F)$,
$C,D$ belong to $\Mat_{p,\nu}(\F)$ and $S \in \Mats_p(\F)$.
Likewise, the $b$-alternating endomorphisms of $V$ are those that are represented in $\bfB$ by a matrix of the form
$$\begin{bmatrix}
A & -(PC)^T & B \\
D & P^{-1} S & C \\
E & -(PD)^T & -A^T
\end{bmatrix}$$
in which $A \in \Mat_\nu(\F)$, $B$ and $E$ belong to $\Mata_\nu(\F)$,
$C$ and $D$ belong to $\Mat_{p,\nu}(\F)$ and $S \in \Mata_p(\F)$.

It is then easily checked that the space
$$\WS_{\nu,P}:=\biggl\{
\begin{bmatrix}
N & (PC)^T & B \\
[0]_{p \times \nu} & 0_p & C \\
0_\nu & [0]_{\nu \times p} & N^T
\end{bmatrix} \mid  N \in \NT_\nu(\F), \; C \in \Mat_{p \times \nu}(\F), \; B \in \Mats_\nu(\F)\biggr\}$$
consists of $S_{\nu,P}$-symmetric nilpotent matrices; moreover its dimension equals
$\nu^2+\nu p=\nu(n-\nu)$.
Next, the space
$$\WA_{\nu,P}:=\Biggl\{
\begin{bmatrix}
N & -(PC)^T & B \\
[0]_{p \times \nu} & 0_p & C \\
0_\nu & [0]_{\nu \times p} & -N^T
\end{bmatrix} \mid  N \in \NT_\nu(\F), \; C \in \Mat_{p \times \nu}(\F), \; B \in \Mata_\nu(\F)\Biggr\}$$
consists of $S_{\nu,P}$-alternating nilpotent matrices, and
one checks that its dimension equals $\nu(\nu-1)+\nu p=\nu(n-\nu-1)$.

\subsection{Examples: the geometric viewpoint}\label{geomExamples}

We now turn to a more geometric viewpoint of the above matrix spaces.
We start with two very basic lemmas on the properties of $b$-symmetric and $b$-alternating endomorphisms.

\begin{lemma}\label{stableortholemma}
Let $b$ be a non-degenerate symmetric or alternating bilinear form on $V$, and let
$u\in \calS_b \cup \calA_b$.
Then:
\begin{enumerate}[(a)]
\item For every linear subspace $W$ of $V$ that is stable under $u$, the subspace
$W^\bot$ is stable under $u$.
\item One has $\Ker u=(\im u)^\bot$.
\end{enumerate}
\end{lemma}

\begin{proof}
Remember that there exists an $\epsilon' \in \{-1,1\}$ such that
$$\forall (x,y)\in V^2, \quad b(x,u(y))=\epsilon'\, b(u(x),y).$$
Let $W$ be a linear subspace of $V$ that is stable under $u$.
Let $x \in W^\bot$. For all $y \in W$, we have $b(u(x),y)=\epsilon'\, b(x,u(y))=0$, whence $u(x) \in W^\bot$.
This proves point (a).

Next, for all $x \in V$, we see that
$x \in (\im u)^\bot$ is successively equivalent to $\forall y \in V, \; b(x,u(y))=0$, to
$\forall y \in V, \; b(u(x),y)=0$, and finally to $u(x)=0$ since $b$ is non-degenerate.
This yields point (b).
\end{proof}

\begin{lemma}\label{nonisotropicLemma}
Let $b$ be a non-isotropic symmetric bilinear form on $V$, and
$u$ be a nilpotent $b$-symmetric or $b$-alternating endomorphism of $V$.
Then, $u=0$.
\end{lemma}

\begin{proof}
Indeed, $\Ker u=(\im u)^\bot$ by Lemma \ref{stableortholemma}. Since $b$ is non-isotropic,
it follows that $\im u$ is a complementary subspace of $\Ker u$. Thus, $u$ induces
an automorphism of $\im u$; yet this induced endomorphism should be nilpotent. Therefore
$\im u=\{0\}$, which shows that $u=0$.
\end{proof}

\begin{Def}
Remember that a flag of $V$ is an increasing list $\calF=(F_0,\dots,F_p)$
of linear subspaces of $V$. We call such a flag \textbf{partially complete} if $\dim F_i=i$
for all $i \in \lcro 0,p\rcro$. Given a bilinear form $b$ on $V$ (symmetric or alternating), we call such a flag
\textbf{$b$-singular} if $F_p$ is totally $b$-singular.
A partially complete $b$-singular flag $(F_0,\dots,F_p)$ of $V$ is called \textbf{maximal} if $p$ equals the Witt index of $b$.
\end{Def}

To construct a partially complete $b$-singular flag of $V$, it suffices to
start from a totally $b$-singular subspace $G$ of $V$ and to take a complete flag of $G$.

\begin{prop}\label{dimspecialprop}
Let $b$ be a non-degenerate symmetric or alternating bilinear form on a vector space $V$ with dimension $n$, with Witt index $\nu$. Let $\calF$ be a maximal partially complete $b$-singular flag of $V$. Then:
\begin{enumerate}[(a)]
\item The set $\WS_{b,\calF}$ of all nilpotent endomorphisms
$u \in \calS_b$ that stabilize $\calF$ is a linear subspace of dimension $\nu(n-\nu)$.
\item The set $\WA_{b,\calF}$ of all nilpotent endomorphisms
$u \in \calA_b$ that stabilize $\calF$ is a linear subspace of dimension $\nu(n-\nu-1)$.
\end{enumerate}
\end{prop}

Note here that the underlying field is totally arbitrary, possibly of characteristic $2$.

In order to prove the above result, we will transfer the problem to the matrix setting, and will then solve it with a method
similar to the one from the previous section.
Let us set $p:=n-2\nu$, and say that a basis $(e_1,\dots,e_n)$ of $V$ is \textbf{adapted} to $\calF=(F_0,\dots,F_\nu)$
whenever the following conditions are satisfied:
\begin{itemize}
\item $(e_1,\dots,e_i)$ is a basis of $F_i$ for all $i \in \lcro 0,\nu\rcro$;
\item $e_{\nu+1},\dots,e_{\nu+p}$ are orthogonal to $e_1,\dots,e_\nu,e_{n-\nu+1},\dots,e_n$;
\item $b(e_i,e_{n-\nu+j})=\delta_{i,j}$ for all $(i,j) \in \lcro 1,\nu\rcro^2$.
\end{itemize}
A basis that is adapted to $\calF$ is called \textbf{strongly adapted} to $\calF$
if, in addition to the above requirements, the subspace $\Vect(e_k)_{n-\nu+1 \leq k \leq n}$ is totally $b$-singular.

The existence of an adapted basis is folklore and we quickly recall the main arguments that justify it:
\begin{itemize}
\item One chooses a basis $(e_1,\dots,e_\nu)$ of $F_\nu$ such that $F_i=\Vect(e_1,\dots,e_i)$ for all $i \in \lcro 1,\nu\rcro$.
\item One splits $F_\nu^\bot=F_\nu \oplus G$. Then, $G$ is $b$-regular (better still, $b$ is non-isotropic on $G$), whence $G^\bot$ is $b$-regular.
\item One chooses an arbitrary basis $(g_1,\dots,g_p)$ of $G$.
\item One then splits $G^\bot=F_\nu \oplus H$, so that $\dim H=\dim F$.
As $G^\bot$ is $b$-regular and $F_\nu$ is totally singular, we find a unique basis $(f_1,\dots,f_\nu)$ of $H$
such that $b(e_i,f_j)=\delta_{i,j}$ for all $(i,j)\in \lcro 1,\nu\rcro^2$.
\end{itemize}
The resulting family $(e_1,\dots,e_\nu,g_1,\dots,g_p,f_1,\dots,f_\nu)$ is then a basis of $V$ that is adapted to $\calF$.

Moreover, if $b$ is alternating or if $\F$ does not have characteristic $2$, then there exists a basis that is
strongly adapted to $\calF$: to see this, it suffices to modify the previous construction by choosing $H$, among the direct factors
of $F_\nu$ in $G^\bot$, as a totally $b$-singular subspace (the existence of such a subspace is folklore in the said cases, but can fail
for a symmetric bilinear form over a field with characteristic $2$).

\begin{proof}[Proof of Proposition \ref{dimspecialprop}]
We choose a basis $(e_1,\dots,e_n)$ of $V$ that is adapted to the flag $\calF$, as in the above.
We set $G:=\Vect(e_{\nu+1},\dots,e_{n-\nu})$ and $H:=\Vect(e_{n-\nu+1},\dots,e_n)$, and we see
that $F_\nu^\bot=F_\nu \oplus G$ and $G^\bot=F_\nu\oplus H$ (in both cases, the space on the right-hand side is included in the one on the left-hand side, and the equality of dimensions is clear). The restriction of $b$ to $G\times G$ is equivalent to the bilinear form induced by
$b$ on $F_\nu^\bot/F_\nu$, which is non-isotropic because of the definition of the Witt index.
We denote by $P$ the non-isotropic (necessarily symmetric) matrix of $b_{|G \times G}$ in $(f_1,\dots,f_p)$.

The matrix of $b$ in $\bfB$
reads
$$S_{\nu,P,Q}:=\begin{bmatrix}
0_\nu & [0]_{\nu \times p} & I_\nu \\
[0]_{p \times \nu} & P & [0]_{p \times \nu} \\
\varepsilon\,I_\nu & [0]_{\nu \times p} & Q
\end{bmatrix}$$
for some $Q \in \Mat_\nu(\F)$ (symmetric or alternating), where $\varepsilon=1$ if $b$ is symmetric, and $\varepsilon=-1$
if $b$ is alternating.

Now, let $u \in \calS_b \cup \calA_b$ be nilpotent, and assume that it stabilizes $F_\nu$. Then,
$u$ also stabilizes $F_\nu^{\bot}$ (see Lemma \ref{stableortholemma}), and it follows that it induces a nilpotent endomorphism $\overline{u}$ of the quotient space $F_\nu^\bot/F_\nu$, and $\overline{u}$ is either $\overline{b}$-symmetric or $\overline{b}$-alternating.
By Lemma \ref{nonisotropicLemma}, we deduce that $\overline{u}=0$, i.e.\ $u$ maps $F_\nu^{\bot}$ into $F_\nu$.

Hence, the elements of $\WS_{b,\calF}$ are the nilpotent $b$-symmetric endomorphisms of $V$
that are represented in $\bfB$ by a matrix
of the form
$$M=\begin{bmatrix}
A & C' &  B \\
[0]_{p \times \nu} & 0_p & C \\
0_\nu & [0]_{\nu \times p} & A'
\end{bmatrix}$$
where $A$ is upper-triangular, $C' \in \Mat_{\nu,p}(\F)$, $C \in \Mat_{p,\nu}(\F)$,
$A' \in \Mat_\nu(\F)$ and $B \in \Mat_\nu(\F)$.
A straightforward computation shows that a matrix of the said form is $S_{\nu,P,Q}$-symmetric
if and only if $A'=\varepsilon A^T$, $C'=\epsilon\,(PC)^T$ and there exists
a matrix $E \in \Mats_\nu(\F)$ such that $B=E-QA^T$.
Assuming that $M$ is $S_{\nu,P,Q}$-symmetric, we see that it is nilpotent if and only if $A$ and
$A'$ are nilpotent, which is equivalent to having $A$ nilpotent.
It follows that $\WS_{b,\calF}$ is the set of all endomorphisms of $V$
that are represented in $\bfB$ by a matrix of the form
$$\begin{bmatrix}
A & \epsilon\,(PC)^T & E-QA^T \\
[0]_{p \times \nu} & 0_p & C \\
0_\nu & [0]_{\nu \times p} & \epsilon A^T
\end{bmatrix}$$
for some $A \in \NT_\nu(\F)$, some $C \in \Mat_{p,\nu}(\F)$ and some $E \in \Mats_\nu(\F)$.
Obviously, this is a linear subspace of $\End(V)$ with dimension
$$\dim \NT_\nu(\F)+\dim \Mat_{p,\nu}(\F)+\dim \Mats_\nu(\F)=\nu(n-\nu).$$
Likewise, one proves that
$\WA_{b,\calF}$ is the set of all endomorphisms of $V$
that are represented in $\bfB$ by a matrix of the form
$$\begin{bmatrix}
A & -\epsilon\,(PC)^T & E+QA^T \\
[0]_{p \times \nu} & 0_p & C \\
0_\nu & [0]_{\nu \times p} & -\epsilon A^T
\end{bmatrix}$$
for some $A \in \NT_\nu(\F)$, some $C \in \Mat_{p,\nu}(\F)$ and some $E \in \Mata_\nu(\F)$.
Obviously, this is a linear subspace of $\End(V)$ with dimension
$$\dim \NT_\nu(\F)+\dim \Mat_{p,\nu}(\F)+\dim \Mata_\nu(\F)=\nu(n-\nu-1).$$
\end{proof}

\subsection{Main results: the non-degenerate case}

As we have seen in Section \ref{ProblemSection}, our problem can be reduced to the case when the bilinear form under consideration is
non-degenerate. The following theorem gives a complete answer to our first question in that situation, provided that the underlying field
be of characteristic different from $2$.

\begin{theo}[Main theorem]\label{TheoMajo}
Let $V$ be a finite-dimensional vector space over a field $\F$ with characteristic not $2$, and $b$ be a non-degenerate symmetric or alternating bilinear form on $V$. Denote by $\nu$ the Witt index of $b$, and set $n:=\dim V$.
\begin{enumerate}[(a)]
\item The greatest possible dimension for a nilpotent subspace of $\calS_b$
is $\nu (n-\nu)$.
\item The greatest possible dimension for a nilpotent subspace of $\calA_b$
is $\nu (n-\nu-1)$.
\end{enumerate}
\end{theo}

We have already seen that the dimension bounds in this theorem are optimal.
The results are substantially different for fields with characteristic $2$
and require different techniques: we will deal with them in a subsequent article.

\vskip 3mm
Next, we give several results on the case of equality:
first of all, there are two situations where our results encompass all fields with characteristic not $2$:

\begin{theo}\label{EqualitySymSymtheo}
Let $b$ be a non-degenerate symmetric bilinear form on an $n$-dimensional vector space $V$, over a field with characteristic not $2$. Denote by $\nu$ the Witt index of $b$, and let $\calV$ be a nilpotent linear subspace of $\calS_b$ with dimension $\nu(n-\nu)$.
Then, there exists a maximal partially complete $b$-singular flag $\calF$ of $V$ such that $\calV=\WS_{b,\calF}$.
\end{theo}

\begin{theo}\label{EqualityAltAlttheo}
Let $b$ be a non-degenerate alternating bilinear form on an $n$-dimensional vector space $V$, over a field with characteristic not $2$. Denote by $\nu$ the Witt index of $b$, and let $\calV$ be a nilpotent linear subspace of $\calA_b$ with dimension $\nu(n-\nu-1)$.
Then, there exists a maximal partially complete $b$-singular flag $\calF$ of $V$ such that $\calV=\WA_{b,\calF}$.
\end{theo}

Our proof of the above two theorems cannot be adapted to the other cases in the structured Gerstenhaber problem:
indeed, one key point that we will use is that $\calS_b$ is stable under squares if $b$ is symmetric,
and so is $\calA_b$ if $b$ is alternating. However, $\calS_b$ is in general unstable under squares
if $b$ is alternating, and $\calA_b$ is in general unstable under squares if $b$ is symmetric.
For those two cases, we conjecture that the corresponding results hold.

\begin{conj}\label{EqualityAltSymconj}
Let $b$ be a non-degenerate alternating bilinear form on an $n$-dimensional vector space $V$, over a field with characteristic not $2$.
Set $\nu:=\frac{n}{2}$, and let $\calV$ be a nilpotent linear subspace of $\calS_b$ with dimension $\nu(n-\nu)$.
Then, there exists a maximal partially complete $b$-singular flag $\calF$ of $V$ such that $\calV=\WS_{b,\calF}$.
\end{conj}

\begin{conj}\label{EqualitySymAltconj}
Let $b$ be a non-degenerate symmetric bilinear form on an $n$-dimensional vector space $V$, over a field with characteristic not $2$.
Denote by $\nu$ the Witt index of $b$, and let $\calV$ be a nilpotent linear subspace of $\calA_b$ with dimension $\nu(n-\nu-1)$.
Then, there exists a maximal partially complete $b$-singular flag $\calF$ of $V$ such that $\calV=\WA_{b,\calF}$.
\end{conj}

We also conjecture that Theorem \ref{EqualityAltAlttheo} can be generalized to all fields with characteristic $2$.

We will not tackle the above two conjectures here, but we can already announce that they hold under the additional assumption that the underlying field be of large enough cardinality with respect to the Witt index of $b$. The proofs will be given in a subsequent article.

\vskip 3mm
At this point, it should be noted that, among the above results, some were already known prior to our study but in very special cases only.
First of all, the dimensions bounds from Theorem \ref{TheoMajo}
were first found by Meshulam and Radwan \cite{MeshulamRadwan} in the special case of the field of complex numbers for a symmetric form: Meshulam and Radwan consider the standard bilinear form $b : (X,Y) \mapsto X^TY$ on $\C^n$,
in which case $\calS_b$ and $\calA_b$ correspond to the matrix spaces $\Mats_n(\C)$ and $\Mata_n(\C)$, respectively.
Their proofs can easily be generalized to an algebraically closed field with characteristic not $2$ for the former, and
an algebraically closed field with characteristic $0$ for the latter.
The optimal upper bound was recently rediscovered by Bukov\v{s}ek and Omladi\v{c} \cite{BukovsekOmladic} for
symmetric complex matrices: they obtained Theorem \ref{EqualitySymSymtheo} in the special case of the field of complex numbers, but their
proof can be generalized to an arbitrary algebraically closed field with characteristic not $2$; they rediscovered some ideas
of Meshulam and Radwan \cite{MeshulamRadwan} and mixed them with an adaptation of the acclaimed proof of Gerstenhaber's theorem by Mathes, Omladi\v{c} and Rajdavi \cite{Mathes}, which connects it to the famous Jacobson triangularization theorem
\cite{Jacobson,Radjavi} for sets of nilpotent endomorphisms.

On the other hand, when $b$ has maximal Witt index among the non-degenerate forms (i.e. $\nu=\lfloor \frac{n}{2}\rfloor$)
and the underlying field is algebraically closed, the dimension bounds from Theorem \ref{TheoMajo} are known, for alternating forms in point (a), and for symmetric forms in point (b), as a special case of the dimension bound obtained by Draisma, Kraft, and Kuttler for subspaces of nilpotent
elements of a reductive Lie algebra \cite{DraismaKraftKuttler}.

\subsection{Main results: the general case}

Combining the results of the preceding section with the standard Gerstenhaber theorem,
we can give a full answer to the dimension bound problem in the structured Gerstenhaber theorem, as well as
the classification of spaces of maximal dimension in the situations that correspond to Theorems \ref{EqualitySymSymtheo} and Theorem \ref{EqualityAltAlttheo}.
We simply state the results: their proofs are mostly straightforward by using the above theorems, Gerstenhaber's theorem
and the reduction technique that is discussed in the end of Section \ref{ProblemSection}.
Note to this end that if $r$ denotes the rank of $b$, then the dimension of the radical of $b$ equals $n-r$, and the Witt index of the reduced non-degenerate form $\overline{b}$ equals $\nu-n+r$.

\begin{theo}
\label{TheoMajoGeneral}
Let $V$ be a finite-dimensional vector space over a field with characteristic not $2$, and $b$ be a symmetric or alternating bilinear form
on $V$. Denote by $r$ the rank of $b$, by $\nu$ the Witt index of $b$, and set $n:=\dim V$.
\begin{enumerate}[(a)]
\item The greatest possible dimension for a nilpotent linear subspace of $\calS_b$
is $\dbinom{n-r}{2}+r(n-r)+(\nu-n+r) (n-\nu)$.
\item The greatest possible dimension for a nilpotent linear subspace of $\calA_b$
is $\dbinom{n-r}{2}+r(n-r)+(\nu-n+r) (n-\nu-1)$.
\end{enumerate}
\end{theo}

\begin{theo}\label{EqualitySymSymtheoGeneral}
Let $b$ be a symmetric bilinear form on an $n$-dimensional vector space $V$, over a field with characteristic not $2$. Denote by $\nu$ the Witt index of $b$, by $r$ its rank, and let $\calV$ be a nilpotent linear subspace of $\calS_b$ with dimension
$\dbinom{n-r}{2}+r(n-r)+(\nu-n+r) (n-\nu)$.
Then, there exists a maximal partially complete $b$-singular flag $\calF=(F_0,\dots,F_\nu)$ of $V$ such that
$F_{n-r}=\Rad(b)$ and $\calV$ be the set of all nilpotent $b$-symmetric endomorphisms of $V$ that stabilize $\calF$.
\end{theo}

\begin{theo}\label{EqualityAltAlttheoGeneral}
Let $b$ be an alternating bilinear form on an $n$-dimensional vector space $V$, over a field with characteristic not $2$. Denote by $\nu$ the Witt index of $b$, by $r$ its rank, and let $\calV$ be a nilpotent linear subspace of $\calA_b$ with dimension
$\dbinom{n-r}{2}+r(n-r)+(\nu-n+r) (n-\nu-1)$.
Then, there exists a maximal partially complete $b$-singular flag $\calF=(F_0,\dots,F_\nu)$ of $V$ such that
$F_{n-r}=\Rad(b)$ and $\calV$ be the set of all nilpotent $b$-alternating endomorphisms of $V$ that stabilize $\calF$.
\end{theo}

\subsection{Strategy, and structure of the article}

The present article is the first entry in a series of articles on the structured Gerstenhaber problem.
It has two ambitions: firstly, to serve as an introduction to the problem; secondly, to prove all the
results that can be obtained, with limited effort, by using the strategy of Mathes, Omladi\v{c} and Radjavi \cite{Mathes}; thirdly,
to pave the way for a resolution of Conjectures \ref{EqualityAltSymconj} and \ref{EqualitySymAltconj}.
When those results will be obtained, this will not be the end of the story at all, since there will remain:
\begin{itemize}
\item To tackle fields with characteristic $2$.
\item To tackle Conjectures \ref{EqualityAltSymconj} and \ref{EqualitySymAltconj}.
\end{itemize}
That extra work will prove to be much more difficult than what is featured in the present article, and
will be dealt with in a series of subsequent papers.

\vskip 3mm
As we have just said, our proofs of Theorems \ref{TheoMajo}, \ref{EqualitySymSymtheo} and \ref{EqualityAltAlttheo} are adaptations of the celebrated Mathes-Omladi\v{c}-Radjavi method,
which uses trace orthogonality techniques to both majorize the dimension of a nilpotent subspace and prove that any space with the maximal dimension is stable under squares. The trace method fails in the characteristic $2$ case, which explains that in this article we completely discard it except when Hermitian forms are involved (in that case it is commonplace that fields with characteristic $2$ do not carry any special difficulty).

The remainder of the article is organized as follows:

\begin{itemize}
\item In Section \ref{PrelimSection}, we recall the generalized trace orthogonality lemma
(Lemma \ref{tracelemma})  and give a new proof of it. We also prove a basic result on the triangularization of a nilpotent $b$-symmetric or $b$-alternating endomorphism.
\item In the next three sections, we give three different proofs of Theorem \ref{TheoMajo}.
The first one (Section \ref{InductiveSection}) is an inductive proof that emphasizes the use of elements of small rank in $\calS_b$ and $\calA_b$. The second one (Section \ref{DirectGerstenSection}) is a direct proof that relies upon Gerstenhaber's theorem.
The third one (Section \ref{DirectMORSection}) is a direct self-contained proof: it is an adaptation
of the Mathes-Omladi\v{c}-Radjavi proof of Gerstenhaber's theorem.
The point of giving three different proofs is that each one of them can be used to give a specific piece of
information on the structure of the spaces with maximal dimension. Collecting such information will probably
be of great interest in the prospect of proving Conjectures \ref{EqualityAltSymconj} and \ref{EqualitySymAltconj}.
\item Following the third proof of Theorem \ref{TheoMajo}, we apply the Mathes-Omladi\v{c}-Radjavi strategy in Section \ref{MaxDimSection} to obtain Theorems \ref{EqualitySymSymtheo} and \ref{EqualityAltAlttheo}: there, we combine information that is given by the inductive proof
with information that follows from the direct proof of Section \ref{DirectMORSection}.
\item Finally, Section \ref{HermitianSection} is devoted to the structured Gerstenhaber problem in the Hermitian case. There, the form $b$ is Hermitian with respect to a non-identity involution of $\F$, and we consider spaces of $b$-Hermitian endomorphisms.
\end{itemize}

\section{Preliminary results}\label{PrelimSection}

\subsection{General results on nilpotent spaces of endomorphisms}

We will need the following basic lemma. It was first proved by MacDonald, MacDougall and Sweet \cite{Macs}
under slightly stronger assumptions:

\begin{lemma}[Generalized trace lemma]\label{tracelemma}
Let $A$ and $B$ be matrices of $\Mat_n(\F)$, and $k$ be a non-negative integer.
Assume that the matrix $\lambda A+\mu B$ is nilpotent for at least $k+2$ pairwise linearly independent pairs $(\lambda,\mu) \in \F^2$.
Then, $\tr(A^k B)=0$.
\end{lemma}

Our proof is a variation of the one from \cite{Macs}. Given $M \in \Mat_n(\F)$, we write the characteristic polynomial of $M$ as
$$\chi_M(t)=\det(tI_n-M)=\sum_{k=0}^n c_k(M)\, t^{n-k}$$
so that, by setting  $c_k(M):=0$ whenever $k>n$, we find
$$\det(I_n-tM)=\sum_{k=0}^{+\infty} c_k(M)\, t^k.$$
The above lemma will then be seen as a consequence of the following one:

\begin{lemma}\label{difflemma}
Let $A$ and $B$ be matrices of $\Mat_n(\F)$.
For every positive integer $k$, the polynomial $c_k(A+sB)$ of $\F[s]$ has derivative $-\underset{i=0}{\overset{k-1}{\sum}} c_{k-1-i}(A) \tr(A^i B)$ at zero.
\end{lemma}

\begin{proof}
We consider $A$ and $B$ as matrices over the field $\F((s,t))$ of fractions of the power series ring in two variables $s$ and $t$ with coefficients in $\F$.
Differentiating the identity
$$\det (I_n-t(A+sB))=\sum_{k=0}^{+\infty} c_k(A+sB)\, t^k$$
at $0$ with respect to $s$ yields
$$\tr\bigl((I_n-tA)^\ad (-tB)\bigr)= \sum_{k=0}^{+\infty} \frac{\diff c_k(A+sB)}{\diff s}\Bigr|_{s=0}\, t^k,$$
where $M^\ad$ denotes the principal adjoint of the square matrix $M$ (i.e.\ the transpose of the comatrix of $M$).

The left-hand side of that equality can be rewritten
\begin{align*}
\tr\bigl((I_n-tA)^\ad (-tB)\bigr) & = -t \det(I_n-tA) \tr((I_n-tA)^{-1} B) \\
& =-t \det(I_n-tA) \sum_{k=0}^{+\infty} \tr(A^k B)\, t^k \\
& = -t \biggl(\sum_{i=0}^{+\infty} c_i(A)\, t^i\biggr) \biggl(\sum_{k=0}^{+\infty} \tr(A^k B)\, t^k\biggr) \\
& = -\sum_{k=1}^{+\infty}\biggl(\sum_{i=0}^{k-1} c_{k-1-i}(A) \tr(A^iB)\biggr) t^k.
\end{align*}
The claimed equalities follow.
\end{proof}

\begin{proof}[Proof of Lemma \ref{tracelemma}]
Let $p \in \lcro 0,k\rcro$.
Note that $c_{p+1}(tA+sB)$ is a homogeneous polynomial of $\F[s,t]$ with degree $p+1$.
By assumption, it vanishes at $k+2$ elements (at least) of the projective space $\Pgros(\F^2)$, whence
it is identically zero. In particular, its partial derivative with respect to $s$ at $(t,s)=(1,0)$ equals zero.
By Lemma \ref{difflemma}, this yields $\underset{i=0}{\overset{p}{\sum}} c_{p-i}(A) \tr(A^i B)=0$.
Noting that $c_0(A)=1$, we deduce by induction that $\tr(A^p B)=0$ for all $p \in \lcro 0,k\rcro$.
\end{proof}

As a corollary of the generalized trace lemma, we have:

\begin{lemma}[Trace Orthogonality Lemma]\label{tracecor}
Let $u,v$ be endomorphisms of a finite-dimensional vector space over $\F$.
\begin{enumerate}[(a)]
\item If every linear combination of $u$ and $v$ is nilpotent, then $\tr(uv)=0$.
\item Given a positive integer $k$, if $v$ is nilpotent and there are at least $k+1$ values of $\lambda$ in $\F$
for which $u+\lambda v$ is nilpotent, then $\tr(u^kv)=0$.
\end{enumerate}
\end{lemma}

\subsection{Stable flags for a nilpotent $b$-symmetric or $b$-alternating endomorphism}

In the analysis of the nilpotent spaces with maximal dimension, the following basic result will be useful.
Note that it holds regardless of the characteristic of the underlying field:

\begin{lemma}\label{stableflaglemma}
Let $u$ be a nilpotent $b$-symmetric or $b$-alternating endomorphism of $V$.
Then, there exists a maximal partially complete $b$-singular flag $\calF$ of $V$
that is stable under $u$.
\end{lemma}

\begin{proof}
We proceed by induction on the dimension of $V$. The result is obvious if $u=0$. Assume now that $u \neq 0$, and denote by $\nu$
the Witt index of $b$.
Since $u$ is nilpotent, $\Ker u \cap \im u \neq \{0\}$ and we choose a non-zero vector $x \in \Ker u \cap \im u$. Then, $u$ stabilizes the linear hyperplane $\{x\}^\bot$, which includes $\F x$.
The bilinear form $b$ induces a bilinear form $\overline{b}$ (either symmetric or alternating)
on $\{x\}^\bot/\F x$ with Witt index $\nu-1$.
The endomorphism $u$ then induces an endomorphism of $\{x\}^\bot/\F x$ that is
$\overline{b}$-symmetric or $\overline{b}$-alternating. By induction,
there is a maximal partially complete $\overline{b}$-singular flag $(G_0,\dots,G_{\nu-1})$ of $\{x\}^\bot/\F x$
that is stable under $\overline{u}$. For all $k \in \lcro 1,\nu\rcro$, denote by $F_k$
the inverse image of $G_{k-1}$ under the canonical projection of $\{x\}^\bot$ onto $\{x\}^\bot/\F x$.
Then, $(\{0\},F_1,\dots,F_\nu)$ is a maximal partially complete $b$-singular flag of $V$
that is stable under $u$.
\end{proof}

\section{The maximal dimension: proof by induction}\label{InductiveSection}

Throughout this section, $b$ denotes a non-degenerate symmetric or alternating bilinear form on a finite-dimensional vector space $V$, over a field $\F$ with characteristic not $2$.

In this section, we give a geometric proof of Theorem \ref{TheoMajo}, by induction on the dimension of the space $V$.
This proof can be viewed as an adaptation of a classical inductive proof of Gerstenhaber's theorem
(see e.g.\ \cite{Macs}). The main tool is the consideration of elements of small rank in $\calS_b$ and $\calA_b$.
In the standard Gerstenhaber problem, one considers elements of rank $1$. Here, elements of rank $1$ or $2$ will be considered:
the next paragraph consists of a systematic study of such elements.

\subsection{Symmetric and alternating tensors}

Let $x$ and $y$ be vectors of $V$.
The endomorphism
$$z \in V \mapsto b(y,z)\,x+b(x,z)\,y$$
is easily shown to be $b$-symmetric: we call it the \textbf{$b$-symmetric} tensor of $x$ and $y$, and we denote it by
$x \otimes_b y$.
We note that $\otimes_b$ is a symmetric bilinear mapping from $V^2$ to $\calS_b$.

Likewise, the endomorphism
$$z \in V \mapsto b(y,z)\,x-b(x,z)\,y$$
is easily shown to be $b$-alternating: we call it the \textbf{$b$-alternating} tensor of $x$ and $y$, and we denote it by
$x \wedge_b y$.
We note that $\wedge_b$ is an alternating bilinear mapping from $V^2$ to $\calA_b$.

One checks that, for all $x \in V \setminus \{0\}$, the linear mapping
$y \mapsto x \otimes_b y$ is injective, whereas $y \mapsto x \wedge_b y$ has its kernel equal to $\F x$.
Given a linear subspace $L$ of $V$, we denote by $x \otimes_b L$ (respectively, by $x \wedge_b L$) the space of all tensors
$x \otimes_b y$ (respectively, $x \wedge_b y$) with $y \in L$.

Now, assume furthermore that $b(x,x)=0$ and $b(x,y)=0$, and let $u$ denote $x \otimes_b y$ or $x \wedge_b y$.
Note that $u(x)=0$ and $u$ maps $\{x\}^\bot$ (which includes $\Vect(x,y)$) into $\F x$.
Then, $\im u \subset \Vect(x,y)$ and $u(\Vect(x,y)) \subset \F x$, whence $\im u^2 \subset \F x$ and $\im u^3=\{0\}$.
Thus, in that case $u$ is nilpotent (with nilindex at most $3$).
Now, we prove a converse statement:

\begin{prop}\label{caractensors}
Let $x \in V$ be a non-zero vector such that $b(x,x)=0$.
Let $u \in \calS_b$ (respectively, $u \in \calA_b$) be such that $u(x)=0$ and $u(\{x\}^\bot) \subset \F x$.
Then, there exists $y \in \{x\}^\bot$ such that $u=x \otimes_b y$ (respectively, $u=x \wedge_b y$).
\end{prop}

\begin{proof}
Since $x \in \Ker u$, we have $\im u \subset \{ x\}^\bot$.
Using $u(\{x\}^\bot) \subset \F x$, we get that either $u(\{x\}^\bot)=\F x$, in which case
$\im u$ has dimension at most $2$ and includes $\F x$, or $\{x\}^\bot \subset \Ker u$, in which case
$\im u \subset (\{x\}^\bot)^\bot=\F x$.

The claimed statement is obvious if $u=0$.
Assume now that $u$ has rank $1$. Then, from the starting remarks we find $\im u=\F x$, whence
$u : z \mapsto \varphi(z)\,x$ for some linear form $\varphi$ on $V$.
Since $u(z)=0$ for all $z \in \{x\}^\bot$, we find $\varphi=\lambda\, b(x,-)$ for some $\lambda \in \F$.
Thus $u=\frac{\lambda}{2}\,x \otimes_b x$ and in particular $u$ is $b$-symmetric.
If $u$ is $b$-alternating it follows that $u=0=x \wedge_b x$ (which actually contradicts the assumption that $\rk u=1$).

Assume finally that $u$ has rank $2$. This yields a vector $y \in \{x\}^\bot \setminus \F x$ such that
$\Vect(x,y)=\im u$, and we recover linear forms $\varphi$ and $\psi$ on $V$ such that
$$u : z \mapsto \varphi(z)\,x+\psi(z)\,y.$$
Then, every $z \in \{x\}^\bot$ annihilates $\psi$, whence $\psi=\lambda\, b(x,-)$ for some $\lambda \in \F$.

Replacing $u$ with $u-x \otimes_b (\lambda y)$ (respectively, with $u-x \wedge_b (\lambda y)$) draws us back to the case when $\rk u \leq 1$, and the conclusion follows.
\end{proof}

The next two lemmas are deduced from the Trace Orthogonality Lemma:

\begin{lemma}\label{orthotensorsym}
Let $\calV$ be a nilpotent linear subspace of $\calS_b$.
Let $x$ be a non-zero vector of $V$. Let $u \in \calV$ and $y \in V$  be such that $x \otimes_b y \in \calV$.
Then, $b(u(x),y)=0$.
\end{lemma}

\begin{proof}
Set $v:=x \otimes_b y$. By Lemma \ref{tracecor}, we have $\tr(v \circ u)=0$. Note that
$$(v \circ u) : z \mapsto b\bigl(y,u(z)\bigr)\, x+b\bigl(x,u(z)\bigr)\, y.$$
Since the operator $z \mapsto b(y,u(z))\, x$ has its range included in $\F x$, its trace
is the one of its restriction to $\F x$, that is $b(y,u(x))$. Likewise, the trace of
$z \mapsto b(x,u(z))\, y$ equals $b(x,u(y))$, and we deduce that
$$0=\tr(v \circ u)=b\bigl(y,u(x)\bigr)+b\bigl(x,u(y)\bigr)=2\, b\bigl(y,u(x)\bigr)=\pm 2\,b\bigl(u(x),y\bigr).$$
Since the characteristic of $\F$ is not $2$, this yields the claimed result.
\end{proof}

\begin{lemma}\label{orthotensoralt}
Let $\calV$ be a nilpotent linear subspace of $\calA_b$.
Let $x$ be a non-zero vector of $V$. Let $u \in \calV$ and $y \in V$  be such that $x \wedge_b y \in \calV$.
Then, $b(u(x),y)=0$.
\end{lemma}

\begin{proof}
Set $v:=x \wedge_b y$. By Lemma \ref{tracecor}, we have $\tr(v \circ u)=0$.
Then, as in the proof of Lemma \ref{orthotensorsym}, we obtain
$$0=\tr(v\circ u)=b(y,u(x))-b(x,u(y))=2 \,b(y,u(x))=\pm 2\, b(u(x),y)$$
and we conclude.
\end{proof}

\subsection{The inductive proof}

Remember that $b$ is a non-degenerate symmetric or alternating bilinear form on an $n$-dimensional vector space $V$. We denote by $\nu$ its Witt index.
If $\nu=0$ then $b$ is non-isotropic and we know from Lemma \ref{nonisotropicLemma} that the sole nilpotent element of $\calS_b$ or $\calA_b$ is the zero endomorphism.

Assume now that $\nu>0$, and choose a non-zero isotropic vector $x \in V$.
Let $\calV$ be a nilpotent linear subspace of $\calS_b$ (respectively, of $\calA_b$).
Consider the subspace
$$\calU:=\{u \in \calV : \; u(x)=0\}.$$
Denote by $\overline{b}$ the bilinear form induced by $b$ on $\overline{V}:=\{x\}^\bot/\F x$.
It is symmetric or alternating, and its Witt index equals $\nu-1$.
Let $u \in \calU$. Since $u$ stabilizes $\F x$, it stabilizes its orthogonal complement $\{x\}^\bot$,
whence it induces a nilpotent endomorphism
$\overline{u}$ of $\overline{V}$. The set
$$\calV \modu x:=\{\overline{u} \mid u \in \calU\}$$
is a nilpotent linear subspace of $\calS_{\overline{b}}$ (respectively, of $\calA_{\overline{b}}$).
By Proposition \ref{caractensors}, the kernel of the surjective linear mapping
$$\Phi : u \in \calU \mapsto \overline{u} \in \calV \modu x$$
reads $x \otimes_b L$ (respectively $x \wedge_b L$) for a unique linear subspace $L$ of $\{x\}^\bot$
(respectively, a unique linear subspace $L$ of $\{x\}^\bot$ that contains $x$),
and we have $\dim \Ker \Phi=\dim L$ (respectively, $\dim \Ker \Phi=\dim L-1$).
On the other hand, we set
$$\calV x:=\{u(x) \mid u \in \calV\},$$
so that $\calU$ is the kernel of the surjective linear mapping $u \in \calV \mapsto u(x) \in \calV x$.
Note that $\calV x \cap \F x=\{0\}$ as no element $u \in \calV$ can satisfy $u(x)=x$ (being nilpotent).

Applying the rank theorem twice, we find
$$\dim \calV=\dim L+\dim (\calV x)+\dim (\calV \modu x)$$
(respectively, $\dim \calV=\dim L-1+\dim (\calV x)+\dim (\calV \modu x)$).

Yet, by Lemma \ref{orthotensorsym} (respectively, Lemma \ref{orthotensoralt}), the subspaces $L$ and $\calV x$ are $b$-orthogonal.
Better still, since $x$ is $b$-orthogonal to $L$, we find that $\F x\oplus \calV x$ is $b$-orthogonal to $L$, leading to
$(\dim (\calV x)+1)+\dim L \leq n$, whence
$$\dim (\calV x)+\dim L \leq n-1.$$

Finally, by induction, we have
$$\dim (\calV \modu x) \leq (\nu-1)(n-2-(\nu-1))=(\nu-1)(n-\nu-1)$$
(respectively, $\dim (\calV \modu x) \leq (\nu-1)(n-2-(\nu-1)-1)=(\nu-1)(n-\nu-2)$).

Hence,
$$\dim \calV \leq n-1+(\nu-1)(n-\nu-1)=\nu(n-\nu)$$
(respectively,
$$\dim \calV \leq n-2+(\nu-1)(n-\nu-2)=\nu(n-\nu-1).)$$

\section{The maximal dimension: direct proof using Gerstenhaber's theorem}\label{DirectGerstenSection}

Throughout the section, $\F$ denotes a field with characteristic not $2$, and
$b$ a non-degenerate symmetric or alternating bilinear form on a vector space $V$ (over $\F$)
with finite dimension $n$. Denote by $\nu$ the Witt index of $b$.
Consider an arbitrary maximal totally singular subspace $F$ of $V$, with dimension $\nu$.
Let $\calV$ be a nilpotent subspace of $\calS_b$ (respectively, of $\calA_b$).
Denote by $\calU$ the subspace of all $u \in \calV$ such that $u(F) \subset F$.
Any such $u$ stabilizes $F^\bot$, and hence induces
a nilpotent endomorphism $u_F$ of $F$ and a nilpotent endomorphism
$\overline{u}$ of $F^\bot/F$.
Denote by $\overline{b}$ the bilinear form induced by $b$ on $F^\bot/F$.
Since $F$ is a maximal totally singular subspace for $b$, the quadratic form $x \mapsto \overline{b}(x,x)$
on $F^\bot/F$ is non-isotropic, whence $\overline{u}=0$ for all $u \in \calU$ (by Lemma \ref{nonisotropicLemma}).
In other words, every $u \in \calU$ maps $F^\bot$ into $F$.

Finally, we denote by $\calV_{F}$ the range of the linear mapping
$$\Phi : u \in \calU \mapsto u_F \in \End(F),$$
and we introduce the linear mapping
$$\Psi : u \in \calV \mapsto \bigl(x \mapsto [u(x)]\bigr) \in \calL(F,V/F),$$
where $\calL(F,V/F)$ denotes the space of all linear maps from $F$ to $V/F$.
Note that the kernel of $\Psi$ is precisely $\calU$, and hence the rank theorem applied to $\Phi$ and $\Psi$
yields
$$\dim \calV=\dim \Ker \Phi+\rk \Psi+\dim \calV_{F}.$$

Next, $\calV_F$ is a nilpotent linear subspace of $\End(F)$, and hence by Gerstenhaber's theorem
$$\rk \Phi \leq \dbinom{\nu}{2}.$$
The proof of Theorem \ref{TheoMajo} will then be complete when we establish the following result:

\begin{claim}
If $\calV$ is a subspace of $\calS_b$, then
$$\dim \Ker \Phi+\rk \Psi \leq \nu(n-2\nu)+\dbinom{\nu+1}{2}.$$
If $\calV$ is a subspace of $\calA_b$, then
$$\dim \Ker \Phi+\rk \Psi \leq \nu(n-2\nu)+\dbinom{\nu}{2}.$$
\end{claim}

\begin{proof}
We shall use the Trace Orthogonality Lemma. It is efficient here to think in matrix terms, just like in Section \ref{matrixExamples}.
Set $p:=n-2\nu$.
We choose a basis $(e_1,\dots,e_n)$ of $V$ that is strongly adapted to a complete flag of $F$.
In that basis, the matrix of $b$ reads
$$\begin{bmatrix}
0_\nu & [0]_{\nu \times p} & I_\nu \\
[0]_{p \times \nu} & P & [0]_{p \times \nu} \\
\epsilon I_\nu  & [0]_{\nu \times p} & 0_\nu
\end{bmatrix}$$
for some $\epsilon \in \{1,-1\}$ and some non-isotropic matrix $P \in \GL_p(\F)$ that is
either symmetric or alternating. However, in the latter case $p=0$, hence in any case $P$ is symmetric (possibly void).

Assume first that $\calV \subset \calS_b$.
Every $u \in \calV$ has its matrix in $\bfB$ of the form
$$\begin{bmatrix}
[?]_{\nu \times \nu} & [?]_{\nu \times p} & [?]_{\nu \times \nu} \\
C_1(u) & [?]_{p \times p} & [?]_{p \times \nu} \\
D_1(u) & \epsilon (PC_1(u))^T & [?]_{\nu \times \nu}
\end{bmatrix} \quad \text{with $C_1(u) \in \Mat_{p,\nu}(\F)$ and $D_1(u) \in \Mats_\nu(\F)$.}$$
Every $v \in \Ker \Phi$ has its matrix in $\bfB$ of the form
$$\begin{bmatrix}
0_\nu & \epsilon (PC_2(v))^T & D_2(v) \\
[0]_{p \times \nu} & 0_p & C_2(v) \\
0_\nu  & [0]_{\nu \times p} & 0_\nu
\end{bmatrix} \quad \text{with $C_2(v) \in \Mat_{p,\nu}(\F)$ and $D_2(v) \in \Mats_\nu(\F)$.}$$
In particular, the mapping
$$F_1 : u \in \calV \mapsto \bigl(C_1(u),D_1(u)\bigr) \in \Mat_{p,\nu}(\F) \times \Mats_\nu(\F)$$
induces an isomorphism from $\im \Psi$ to $\im F_1$, and the mapping
$$F_2 : v \in \Ker \Phi \mapsto \bigl(C_2(v),D_2(v)\bigr) \in \Mat_{p,\nu}(\F) \times \Mats_\nu(\F)$$
induces an isomorphism from $\Ker \Phi$ to $\im F_2$.
Finally, for all $u \in \calV$ and all $v \in \Ker \Phi$, the Trace Orthogonality Lemma (Lemma \ref{tracecor}) yields $\tr(uv)=0$, which, as $P$ is symmetric, reads
$$2\epsilon \tr\bigl(C_1(u)^T P C_2(v)\bigr)+\tr\bigl(D_1(u)D_2(v)\bigr)=0.$$
It follows that $\im F_1$ and $\im F_2$ are orthogonal for the symmetric bilinear form
$$\bigl((C,D),(C',D')\bigr) \in \bigl(\Mat_{p,\nu}(\F) \times \Mats_\nu(\F)\bigr)^2 \mapsto
2\epsilon \tr(C^TPC')+\tr(DD').$$
Since $P$ is invertible and the characteristic of $\F$ is not $2$, this bilinear form is non-degenerate, and one concludes that
$$\rk F_1+\rk F_2 \leq \dim\bigl(\Mat_{p,\nu}(\F) \times \Mats_\nu(\F)\bigr)=\nu(n-2\nu)+\dbinom{\nu+1}{2},$$
which yields the first claimed result.

Let us now consider the case when $\calV$ is a subspace of $\calA_b$.
Any $u \in \calV$ has its matrix in the basis $\bfB$ of the form
$$\begin{bmatrix}
[?]_{\nu \times \nu} & [?]_{\nu \times p} & [?]_{\nu \times \nu} \\
C_1(u) & [?]_{p \times p} & [?]_{p \times \nu} \\
D_1(u) & -\epsilon (PC_1(u))^T & [?]_{\nu \times \nu}
\end{bmatrix} \quad \text{where $C_1(u) \in \Mat_{p,\nu}(\F)$ and $D_1(u) \in \Mata_\nu(\F)$.}$$
Any $v \in \Ker \Phi$ has its matrix in $\bfB$ of the form
$$\begin{bmatrix}
0_\nu & -\epsilon (PC_2(v))^T & D_2(v) \\
[0]_{p \times \nu} & 0_p & C_2(v) \\
0_\nu  & [0]_{\nu \times p} & 0_\nu
\end{bmatrix} \quad \text{where $C_2(v) \in \Mat_{p,\nu}(\F)$ and $D_2(v) \in \Mata_\nu(\F)$.}$$
From there, the proof is essentially similar to the previous one, replacing $\Mats_\nu(\F)$
with $\Mata_\nu(\F)$; here the relevant symmetric non-degenerate bilinear form on
$\Mat_{p,\nu}(\F) \times \Mata_\nu(\F)$ is
$$\bigl((C,D),(C',D')\bigr) \in \bigl(\Mat_{p,\nu}(\F) \times \Mata_\nu(\F)\bigr)^2 \mapsto
-2\epsilon \tr(C^TPC')+\tr(DD').$$
\end{proof}

Hence, we have proved Theorem \ref{TheoMajo}.

\section{The maximal dimension: direct self-contained proof}\label{DirectMORSection}

In this section, we give a direct proof of Theorem \ref{TheoMajo}, and
from this proof we derive partial results on the structure of spaces with maximal dimension under mild cardinality assumptions on the underlying field. We start by tackling symmetric endomorphisms in details (Section \ref{DirectSymSection}); the adaptation
to alternating endomorphisms will be briefly discussed in Section \ref{DirectAltSection}.

\subsection{Symmetric endomorphisms}\label{DirectSymSection}

Let $b$ be a non-degenerate symmetric or alternating bilinear form on a vector space $V$ with finite dimension $n$, over a field $\F$
with characteristic not $2$. Denote by $\nu$ the Witt index of $b$.
Let $\calV$ be a nilpotent linear subspace of $\calS_b$.
We fix a maximal partially complete $b$-singular flag $\calF$ of $V$, and we set $p:=n-2\nu$.
We take a basis $\bfB$ of $V$ that is strongly adapted to $\calF$ (see Section \ref{geomExamples}).

In that basis, the matrix of $b$ reads
$$\begin{bmatrix}
0_\nu & [0]_{\nu \times p} & I_\nu \\
[0]_{p\times \nu} & P & [0]_{p \times \nu} \\
\epsilon I_\nu & [0]_{\nu \times p}  & 0_\nu
\end{bmatrix}$$
where $\varepsilon:=1$ if $b$ is symmetric, $\epsilon:=-1$ if $b$ is alternating, and $P \in \GL_p(\F)$
is non-isotropic, and symmetric or alternating. Just like in Section \ref{DirectGerstenSection},
we note that $P$ is actually symmetric (if $b$ is alternating it is the $0$-by-$0$ matrix).
For every $u \in \calS_b$, the matrix of $u$ in $\bfB$ reads
$$M(u)=\begin{bmatrix}
A(u) & \epsilon(PC_2(u))^T & D_2(u) \\
C_1(u) & P^{-1} S(u) & C_2(u) \\
D_1(u) & \epsilon(PC_1(u))^T & \epsilon A(u)^T
\end{bmatrix}$$
where $A(u) \in \Mat_\nu(\F)$, $D_1(u)$, $D_2(u)$ belong to $\Mats_\nu(\F)$, $C_1(u)$ and $C_2(u)$ belong to $\Mat_{p,\nu}(\F)$, and $S(u) \in \Mats_p(\F)$.

As we have seen in Section \ref{geomExamples}, the elements of $\WS_{b,\calF}$ are exactly the endomorphisms $u \in \calS_b$
for which $C_1(u)=0$, $D_1(u)=0$, $S(u)=0$ and $A(u) \in \NT_\nu(\F)$.

To any $u \in \calS_b$, we assign the strictly upper-triangular matrix $I(u) \in \NT_\nu(\F)$ defined as follows:
$$I(u)_{i,j}=\begin{cases}
A(u)_{j,i} & \text{if $i<j$} \\
0 & \text{otherwise.}
\end{cases}$$
Every $u$ in the kernel of
$$\chi : v \in \calV \mapsto \bigl(I(v),C_1(v),D_1(v)\bigr) \in \NT_\nu(\F) \times \Mat_{p,\nu}(\F) \times \Mats_\nu(\F)$$
is such that $A(u)$ is nilpotent and upper-triangular, and hence strictly upper-triangular. It follows that
 $\calV \cap \WS_{b,\calF}$ is precisely the kernel of $\chi$.

Now, let $u \in \WS_{b,\calF}$ and $v \in \calS_b$.
Noting that
$$M(u)=\begin{bmatrix}
A(u) & \epsilon(PC_2(u))^T & D_2(u) \\
[0]_{p \times \nu} & 0_p & C_2(u) \\
0_\nu & [0]_{\nu \times p} & \epsilon A(u)^T
\end{bmatrix},$$
we use the fact that $P$ is symmetric to obtain
$$\tr(v  u)=2\tr\bigl(A(u)A(v)\bigr)+2\epsilon \tr\bigl(C_2(u)^TPC_1(v)\bigr)+\tr\bigl(D_2(u)D_1(v)\bigr)=0.$$
Using the Trace Orthogonality Lemma (Lemma \ref{tracecor}), we deduce that
the range of $\chi$ is $c$-orthogonal to the direct image of $\calV  \cap \WS_{b,\calF}$ under
$$\chi' : u \in \WS_{b,\calF} \mapsto (A(u),C_2(u),D_2(u)) \in \NT_\nu(\F) \times \Mat_{p,\nu}(\F) \times \Mats_\nu(\F)$$
for the symmetric bilinear form
$$c : \begin{cases}
\bigl(\NT_\nu(\F) \times \Mat_{p,\nu}(\F) \times \Mats_\nu(\F)\bigr)^2 & \longrightarrow \F \\
\bigl((N,C,D),(N',C',D')\bigr) & \longmapsto 2\tr(N^TN')+2\epsilon \tr(C^TPC')+\tr(DD').
\end{cases}$$
Using the fact that $P$ is invertible, one checks that $c$ is non-degenerate, and one deduces that
\begin{align*}
\rk \chi+\dim \bigl(\chi'(\calV \cap \WS_{b,\calF})\bigr) & \leq
\dim\bigl(\NT_\nu(\F) \times \Mat_{p,\nu}(\F) \times \Mats_\nu(\F)\bigr) \\
& \leq \dbinom{\nu}{2}+p\nu+\dbinom{\nu+1}{2}=\nu(n-\nu).
\end{align*}
Yet, by the rank theorem,
$$\dim \calV=\rk \chi+\dim \Ker \chi=\rk \chi+\dim(\calV \cap \WS_{b,\calF}),$$
all the while
$$\dim \bigl(\chi'(\calV \cap \WS_{b,\calF})\bigr)=\dim (\calV \cap \WS_{b,\calF})$$
because $\chi'$ is obviously injective.
We conclude that
$$\dim \calV =\rk \chi+\dim \bigl(\chi'\bigl(\calV \cap \WS_{b,\calF}\bigr)\bigr) \leq \nu(n-\nu).$$

We can also draw a powerful result from the above proof in the case when $\calV$ has the critical dimension:

\begin{lemma}\label{doubleortholemma}
Let $\calV$ be a nilpotent subspace of $\calS_b$ with dimension $\nu(n-\nu)$.
Let $v \in \calS_b$ be nilpotent and such that $\forall u \in \calV, \; \tr(uv)=0$.
Then, $v \in \calV$.
\end{lemma}

\begin{proof}
By Lemma \ref{stableflaglemma}, there is a maximal partially complete $b$-singular flag
$\calF$ of $V$ such that $v \in \WS_{b,\calF}$.
Taking a basis $\bfB$ as in the above, and introducing the linear mappings $\chi$ and $\chi'$
together with the bilinear form $c$ attached to that basis, we get that $\chi'(v)$ is $c$-orthogonal
to the range of $\chi$. However, as $\dim \calV=\nu(n-\nu)$, we get from the above proof
that $\chi'(\calV \cap \WS_{b,\calF})$ is the orthogonal complement of $\im \chi$ under
$c$. It follows that $\chi'(v) \in \chi'(\calV \cap \WS_{b,\calF})$. Since $\chi'$
is injective this yields $v \in \calV$.
\end{proof}

Now, we distinguish between two cases:
\begin{itemize}
\item If $b$ is symmetric, then $\calS_b$ is stable under squares because, for all $u \in \calS_b$,
$$\forall (x,y)\in V^2, \; b\bigl(x,u^2(y)\bigr)=b\bigl(u(x),u(y)\bigr)=b\bigl(u^2(x),y\bigr)=b\bigl(y,u^2(x)\bigr).$$
More generally, $\calS_b$ is stable under
any positive power.
\item If $b$ is alternating, then $\calS_b$ is stable under cubes because, for all $u \in \calS_b$,
$$\forall (x,y)\in V^2, \; b\bigl(x,u^3(y)\bigr)=(-1)^3\, b\bigl(u^3(x),y\bigr)=b\bigl(y,u^3(x)\bigr)$$
(more generally, $\calS_b$ is stable under any odd power).
\end{itemize}
Combining this with the Trace Orthogonality Lemma (Lemma \ref{tracecor}) and with Lemma \ref{doubleortholemma}, we recover stability results
for spaces with the maximal dimension:

\begin{lemma}\label{stabsym}
Let $\calV$ be a nilpotent subspace of $\calS_b$ with dimension $\nu(n-\nu)$.
\begin{enumerate}[(a)]
\item If $b$ is symmetric, then $u^2 \in \calV$ for all $u \in \calV$.
\item If $|\F|>3$, then $u^3 \in \calV$ for all $u \in \calV$.
\end{enumerate}
\end{lemma}

\subsection{Alternating endomorphisms}\label{DirectAltSection}

Here, the proof is an easy adaptation of the above one. We simply point out the main differences:
\begin{itemize}
\item The space $\Mats_\nu(\F)$ must be replaced with $\Mata_\nu(\F)$.
\item The constant $\epsilon$ must be replaced with $-\epsilon$.
\item The symmetric matrix $S(u)$ is replaced with an alternating matrix.
\item The form $c$ is replaced with the symmetric bilinear form on
$\NT_\nu(\F) \times \Mat_{p,\nu}(\F) \times \Mata_\nu(\F)$ defined as follows:
$$c'\bigl((N,C,D),(N',C',D')\bigr)=2\tr(N^TN')-2\epsilon \tr\bigl(C^TPC'\bigr)+\tr(DD').$$
\end{itemize}

From there, we obtain, for every nilpotent subspace $\calV$ of $\calA_b$, the inequality
$$\dim \calV \leq \dbinom{\nu}{2}+\nu(n-2\nu)+\dbinom{\nu}{2}=\nu(n-\nu-1).$$

Moreover, by using the same line of reasoning as in the end of the previous section, we obtain
the following result on spaces having the maximal dimension:

\begin{lemma}\label{stabalt}
Let $\calV$ be a nilpotent subspace of $\calA_b$ with dimension $\nu(n-\nu-1)$.
\begin{enumerate}[(a)]
\item If $b$ is alternating, then $u^2 \in \calV$ for all $u \in \calV$.
\item If $|\F|>3$, then $u^3 \in \calV$ for all $u \in \calV$.
\end{enumerate}
\end{lemma}

Here, the difference with the symmetric case comes from the observation that
$\calA_b$ is stable under cubes if $b$ is symmetric, and stable under any power if $b$ is alternating.

\section{Spaces with the maximal dimension}\label{MaxDimSection}

Here, we prove Theorems \ref{EqualitySymSymtheo} and \ref{EqualityAltAlttheo}.
The strategies are essentially similar, so we will give full details only for the former. Throughout, $\F$ denotes a field with characteristic different from $2$.

\subsection{Spaces of symmetric endomorphisms for a symmetric form}

Our first step is the following result:

\begin{prop}\label{EqualitySymSymprop}
Let $b$ be a non-degenerate symmetric bilinear form on a vector space $V$ over $\F$ with finite dimension $n$. Denote by $\nu$ the Witt index of $b$, and let $\calV$ be a nilpotent linear subspace of $\calS_b$ with dimension $\nu(n-\nu)$. Then, $\calV$ is triangularizable.
\end{prop}

\begin{proof}
By point (a) of Lemma \ref{stabsym}, $\calV$ is stable under squares.
Hence, it is stable under the Jordan product: for all $(u,v)\in \calV^2$, we have indeed
$$uv+vu=(u+v)^2-u^2-v^2 \in \calV.$$
Hence, by Jacobson's triangularization theorem \cite{Jacobson,Radjavi}, $\calV$ is triangularizable.
\end{proof}

We are now ready to prove Theorem \ref{EqualitySymSymtheo}. We prove the result by induction on the Witt index of $b$.
The result is obvious if it equals zero.
Now, assume that it is not zero.
By Proposition \ref{EqualitySymSymprop}, there is a non-zero vector $x$ of $V$ that is annihilated by all the vectors of
$\calV$. First of all, we prove that $x$ is isotropic.

Assume on the contrary that $x$ is not isotropic:
then, $V=\{x\}^\bot \oplus \F x$. The bilinear form $b$ induces a symmetric
bilinear form $b'$ on $\{x\}^\bot$, and the Witt index $\nu'$ of $b'$ is at most $\nu$.
Every element $u\in \calV$ stabilizes $\{x\}^\bot$
and hence induces a nilpotent endomorphism $u'$ which is $b'$-symmetric.
The space $\{u' \mid u \in \calV\}$ is a nilpotent subspace of $\End(\{x\}^\bot)$,
and it is isomorphic to $\calV$. We then deduce from Theorem \ref{TheoMajo} that
$$\dim \calV \leq \nu'(n-1-\nu') \leq \nu'(n-\nu') \leq \nu(n-\nu),$$
where the third inequality comes from $\nu' \leq \nu \leq \frac{n}{2}\cdot$
If $\nu'>0$, the second inequality is sharp. Otherwise the third one is sharp. In any case,
we obtain
$$\dim \calV <\nu(n-\nu),$$
which contradicts our assumptions.

Hence, $x$ is isotropic. From there, we use the line of reasoning from Section \ref{InductiveSection}.
Denoting by $\overline{b}$ the symmetric bilinear form induced by $b$ on $\{x\}^\bot/\F x$,
we find that the Witt index of $\overline{b}$ is $\nu-1$.
Every $u \in \calV$ induces a nilpotent $\overline{b}$-symmetric endomorphism $\overline{u}$ of $\{x\}^\bot/\F x$, and we denote by
$\calV \modu x$ the space of all $\overline{u}$ with $u \in \calV$. As we have seen in the end of Section \ref{InductiveSection},
$$\dim (\calV \modu x)=(\nu-1)\bigl((n-2)-(\nu-1)\bigr).$$
Hence, by induction, we find a maximal partially complete $\overline{b}$-singular flag $(G_0,\dots,G_{\nu-1})$ of $\{x\}^\bot/\F x$
that is stable under every element of $\calV \modu x$. For all $k \in \lcro 1,\nu\rcro$,
denote by $F_k$ the inverse image of $G_{k-1}$ under the canonical projection of $\{x\}^\bot$ onto $\{x\}^\bot/\F x$.
We gather that $\calF:=(\{0\},F_1,\dots,F_\nu)$ is a maximal partially complete $b$-singular flag of $V$
that is stable under every element of $\calV$. Hence, we have established the inclusion
$$\calV \subset \WS_{b,\calF.}$$
Since both spaces have dimension $\nu(n-\nu)$ we conclude that they are equal.

\subsection{Spaces of alternating endomorphisms for an alternating form}

Just like in the previous paragraph, we deduce the following result from point (a) of Lemma \ref{stabalt}:

\begin{prop}\label{EqualityAltAltprop}
Let $b$ be a non-degenerate alternating bilinear form on a vector space $V$ with finite dimension $n$ over $\F$. Denote by $\nu$ the Witt index of $b$, and let $\calV$ be a nilpotent linear subspace of $\calA_b$ with dimension $\nu(n-\nu-1)$. Then, $\calV$ is triangularizable.
\end{prop}

From there, the proof of Theorem \ref{EqualityAltAlttheo} is similar to the one of Theorem \ref{EqualitySymSymtheo}.
It is even simpler because any vector of $V$ is $b$-isotropic, and hence once we have
found a non-zero vector that annihilates all the operators in $\calV$, there is no need
to prove that it is $b$-isotropic.

\section{The Hermitian version}\label{HermitianSection}

In this section, we investigate a Hermitian version of the previous theorems.

\subsection{Review of Hermitian endomorphisms}

Let $\F$ be a field equipped with a non-identity
involution $x \mapsto x^\star$. Classically,
$$\K:=\{x \in \F : x^\star=x\}$$
is a subfield of $\F$, and $\F$ is a separable quadratic extension of $\K$. Moreover,
$$\Tr_{\F/\K} : x \mapsto x +x^\star$$
is a surjective $\K$-linear form on $\F$.

Given a matrix $M \in \Mat_{n,p}(\F)$, we set
$$M^\star:=(M_{j,i}^\star)_{1 \leq i \leq p, 1 \leq j \leq n} \in \Mat_{p,n}(\F).$$
A square matrix $M$ is called \textbf{Hermitian} whenever $M=M^\star$.
We denote by $\Math_n(\F)$ the set of all $n$-by-$n$ Hermitian matrices:
it is a $\K$-linear subspace of $\Mat_n(\F)$ with dimension $n^2$ over $\K$.

Let $V$ be a finite-dimensional vector space over $\F$.
A \textbf{sesquilinear} form on $V$ is a function $b : (x,y) \in V^2 \mapsto b(x,y) \in \F$
that is right-linear, i.e. $b(x,-)$ is linear for all $x \in V$,
and left-semilinear, i.e. $b(\alpha x+x',y)=\alpha^\star\, b(x,y)+b(x',y)$
for all $(\alpha,x,x',y)\in \F \times V^3$.
A sesquilinear form on $V$ is called \textbf{Hermitian} whenever
$$\forall (x,y)\in V^2, \; b(y,x)=b(x,y)^\star.$$
By choosing a basis $(e_1,\dots,e_n)$ of $V$ and by assigning to every sesquilinear form
$b$ on $V$ the matrix $(b(e_i,e_j))_{1 \leq i,j \leq n}$, we obtain an isomorphism
from the space of all sesquilinear forms on $V$ to $\Mat_n(\F)$,
and it induces an isomorphism of the $\K$-linear subspace of all Hermitian forms
to the $\K$-linear subspace $\Math_n(\F)$.

Assume that $b$ is Hermitian. Given a subset $X$ of $V$, the right and left-orthogonal complement of $X$ with respect to $b$ are equal, and they are an $\F$-linear subspace of $V$ which we denote by $X^\bot$.
We say that $b$ is non-degenerate whenever $V^\bot=\{0\}$.
In that case and if $X$ is an $\F$-linear subspace of $V$, then $\dim_\F X+\dim_\F X^\bot=\dim_\F V$ and
$(X^\bot)^\bot=V$. We say that $X$ is totally $b$-singular whenever $X \subset X^\bot$.  We say that $b$ is non-isotropic whenever $\forall x \in V, \; b(x,x)=0 \Rightarrow x=0$.
The Witt index of $b$ is defined as the greatest possible dimension for a totally $b$-singular subspace of $V$.

Let $u \in \End_\F(V)$. We say that $u$ is $b$-Hermitian whenever
the sesquilinear form $(x,y) \mapsto b(x,u(y))$ is Hermitian, a condition which easily seen to be equivalent to the identity
$$\forall (x,y)\in V^2, \; b\bigl(u(x),y\bigr)=b\bigl(x,u(y)\bigr).$$
The set of all $b$-Hermitian endomorphisms of $V$ is denoted by $\calH_b$.
It is a $\K$-linear subspace of $\End_\F(V)$, but in general not an $\F$-linear one!
Given a Hermitian matrix $H \in \Math_n(\F)$ and a matrix $M \in \Math_n(\F)$, the
endomorphism $X \mapsto MX$ of $\F^n$ is Hermitian with respect to the Hermitian form
$(X,Y) \mapsto X^\star HY$ if and only if $HM$ is Hermitian: in that situation we say that $M$
is \textbf{$H$-Hermitian}.

Finally, a partially complete $b$-singular flag of $V$ is a flag $(F_0,\dots,F_p)$ of $\F$-linear subspaces of
$V$ such that $\dim F_i=i$ for all $i \in \lcro 0,p\rcro$, and $F_p$ is $b$-singular. Such a flag is called
maximal if $p$ equals the Witt index of $b$.

The following three lemmas are proved in the same manner as Lemmas \ref{stableortholemma},
\ref{nonisotropicLemma} and \ref{stableflaglemma}, respectively:

\begin{lemma}
Let $b$ be a non-degenerate Hermitian form on $V$, and
$u$ be a nilpotent $b$-Hermitian endomorphism of $V$.
Then:
\begin{enumerate}[(a)]
\item For every linear subspace $W$ of $V$ that is stable under $u$, the subspace
$W^\bot$ is also stable under $u$.
\item We have $\Ker u=(\im u)^\bot$.
\end{enumerate}
\end{lemma}

\begin{lemma}\label{nonisotropiclemmaH}
Let $b$ be a non-isotropic Hermitian form on $V$, and
$u$ be a nilpotent $b$-Hermitian endomorphism of $V$.
Then, $u=0$.
\end{lemma}

\begin{lemma}\label{stableflaghermitian}
Let $b$ be a non-degenerate Hermitian form on $V$, and
let $u \in \calH_b$ be nilpotent. Then, some maximal partially complete $b$-singular flag of $V$ is stable under $u$.
\end{lemma}

\subsection{Examples of large spaces of nilpotent $b$-Hermitian endomorphisms}\label{HermitianExample}

Just like in Section \ref{geomExamples}, we can construct large $\K$-linear nilpotent subspaces of $\calH_b$.

\begin{Not}
Let $b$ be a non-degenerate Hermitian form on a finite-dimensional vector space $V$ over $\F$.
Let $\calF$ be a maximal partially complete $b$-singular flag of $V$.
We denote by $\WH_{b,\calF}$ the set of all \emph{nilpotent} $u \in \calH_b$ that stabilize $\calF$.
\end{Not}

\begin{prop}
With the data from the previous definition, $\WH_{b,\calF}$
is a $\K$-linear subspace of $\End_\F(V)$ with dimension
$\nu(2n-2\nu-1)$.
\end{prop}

\begin{proof}
Write $\calF=(F_0,\dots,F_\nu)$, so that $F_\nu$ is totally singular for $b$. As $b$ is Hermitian and non-degenerate,
it is folklore that there exists a complementary subspace $H$ of $F_\nu^\bot$
that is totally $b$-singular. Set $p:=n-2\nu$. As $b$ is non-degenerate, we can choose a basis $(e_1,\dots,e_\nu)$ of $F_\nu$ that is adapted to $\calF$, and then a basis $(e'_1,\dots,e'_\nu)$ of $H$ such that $b(e_i,e'_j)=\delta_{i,j}$ for all $(i,j)\in \lcro 1,\nu\rcro^2$.
The restriction of $b$ to $(F_\nu \oplus H)^2$ is non-degenerate, whence
$G:=(F_\nu \oplus H)^\bot$ is a complementary subspace of $F_\nu \oplus H$
in $V$. We choose a basis $(f_1,\dots,f_p)$ of that space and set
$P:=(b(f_i,f_j))_{1 \leq i,j \leq p}$. Hence, the matrix of $b$ in
the basis $(e_1,\dots,e_\nu,f_1,\dots,f_p,e'_1,\dots,e'_\nu)$ equals
$$H:=\begin{bmatrix}
0_\nu & [0]_{\nu \times p} & I_\nu \\
[0]_{p \times \nu} & P & [0]_{p \times \nu} \\
I_\nu & [0]_{\nu \times p} & 0_\nu
\end{bmatrix}$$
and we have $P \in \Math_p(\F)\cap \GL_p(\F)$.

Let $u \in \WH_{b,\calF}$. Then, $u$ stabilizes $F_\nu$, and hence it also stabilizes $F_\nu^\bot$
and it induces an endomorphism $\overline{u}$ of $F_\nu^\bot/F_\nu$.
The Hermitian form $b$ induces a Hermitian form $\overline{b}$ on $F_\nu^\bot/F_\nu$, and
the definition of $\nu$ yields that $\overline{b}$ is non-isotropic. Obviously $\overline{u}$
is $\overline{b}$-Hermitian, whence Lemma \ref{nonisotropiclemmaH} yields $\overline{u}=0$, so that
$u$ maps $F_\nu^\bot$ into $F_\nu$.

It follows that the set consisting of the matrices in $\bfB$ of
the elements of $\WH_{b,\calF}$ is the set of all $H$-Hermitian nilpotent matrices of the form
$$M=\begin{bmatrix}
A & B &  E \\
[0]_{p \times \nu} & 0_p & C \\
0_\nu & [0]_{\nu \times p} & D
\end{bmatrix}$$
with $A \in \Mat_\nu(\F)$ upper-triangular, $D \in \Mat_\nu(\F)$, $B \in \Mat_{\nu,p}(\F)$, $C \in \Mat_{p,\nu}(\F)$, and $E \in \Mat_\nu(\F)$. Next, a straightforward computation shows that
such a matrix is $H$-Hermitian if and only if $D=A^\star$, $E$ is Hermitian
and $B=(PC)^\star$; in that case, it is nilpotent if and only if $A$ is nilpotent, i.e. $A \in \NT_\nu(\F)$.
It follows that $\WH_{b,\calF}$ is a $\K$-linear subspace of $\calH_b$
that is isomorphic to the $\K$-linear subspace of all matrices of the form
$$\begin{bmatrix}
A & (PC)^\star & E \\
[0]_{p \times \nu} & 0_p & C \\
0_\nu & [0]_{\nu \times p} & A^\star
\end{bmatrix} \quad \text{with $A \in \NT_\nu(\F)$, $C \in \Mat_{p,\nu}(\F)$ and
$E \in \Math_\nu(\F)$.}$$
The claimed result follows since the dimension of the latter space is obviously
$$\nu(\nu-1)+2 p\nu+\nu^2=\nu(2\nu-1+2p)=\nu(2n-2\nu-1).$$
\end{proof}

\subsection{Main result}

Now, we can state our main result on subspaces of nilpotent Hermitian endomorphisms:

\begin{theo}\label{HermitianTheo}
Let $V$ be a vector space with finite dimension $n$ over $\F$, and $b$ be a non-degenerate
Hermitian form on $V$, whose Witt index we denote by $\nu$.
Let $\calV$ be a nilpotent $\K$-linear subspace of $\calH_b$. Then:
$$\dim \calV \leq \nu(2n-2\nu-1).$$
If in addition $|\F|>4$ and $\dim \calV=\nu(2n-2\nu-1)$, then
there exists a maximal partially complete $b$-singular flag $\calF$ of the
$\F$-vector space $V$ such that $\calV=\WH_{b,\calF}$.
\end{theo}

The remainder of the article is devoted to the proof of this theorem.
We will adapt the proof strategy that was applied in the symmetric/alternating case.
We give two proofs of the inequality statement of Theorem \ref{HermitianTheo}, one that
follows the ideas of Section \ref{InductiveSection}, and one that follows the Mathes-Omladi\v{c}-Radjavi strategy.
It is also possible to adapt the direct proof from Section \ref{DirectGerstenSection}, but it would require the skew-field
version of Gerstenhaber's theorem \cite{dSPGerstenhaberskew}.

\subsection{Inductive proof of the inequality in Theorem \ref{HermitianTheo}}

Here, we adapt the inductive proof of Section \ref{InductiveSection}.

Let $b$ be a non-degenerate Hermitian form on an $\F$-vector space $V$ with finite dimension $n$.
Given two vectors $x$ and $y$ of $V$, the endomorphism
$$z \mapsto b(y,z)\,x+b(x,z)\,y$$
is easily seen to be $b$-Hermitian.
We call it the $b$-Hermitian tensor of $x$ and $y$, and we denote it by $x \otimes_b y$.
Note that $(x,y) \mapsto x \otimes_b y$ defines a symmetric $\K$-bilinear mapping from $V^2$ to $\calH_b$.
Given a non-zero vector $x$ of $V$, the $\K$-linear mapping $y \mapsto x \otimes_b y$
has its kernel equal to the $1$-dimensional subspace $\{\lambda x \mid \lambda \in \F \; \text{such that}\; \Tr_{\F/\K}(\lambda)=0\}$.
Given a $\K$-linear subspace $L$ of $V$, we denote by $x \otimes_b L$
the set of all tensors $x \otimes_b y$ with $y \in L$: it is therefore a $\K$-linear subspace of
$\calH_b$.

Now, given $x \in V$ such that $b(x,x)=0$, and given $y \in V$ such that $b(x,y)=0$, one
checks that $x \otimes_b y$ maps $V$ into $\Vect(x,y)$, $\{x\}^\bot$ (which includes $\Vect(x,y)$) into $\F x$, and $\F x$ into $\{0\}$ (and in particular it is nilpotent with nilindex at most $3$). As in the case of symmetric and alternating tensors, there is a converse statement:

\begin{prop}\label{caractensorsHermitian}
Let $x \in V$ be a non-zero vector such that $b(x,x)=0$.
Let $u \in \calH_b$ be such that $u(x)=0$ and $u(\{x\}^\bot) \subset \F x$.
Then, there exists $y \in \{x\}^\bot$ such that $u=x \otimes_b y$.
\end{prop}

\begin{proof}
The claimed statement is obvious if $u=0$.
Assume now that $u$ has rank $1$. Then, as in the proof of Lemma \ref{caractensors}, we find $\im u=\{x\}^\bot$, whence
$u : z \mapsto \varphi(z)\,x$ for some linear form $\varphi$ on $V$.
Since $u(z)=0$ for all $z \in \{x\}^\bot$, we find $\varphi=\lambda\, b(x,-)$ for some $\lambda \in \F$.
Then, since $u$ is $b$-Hermitian, one obtains $\lambda^\star=\lambda$.
It follows that $\lambda=\alpha+\alpha^\star$ for some $\alpha \in \F$, and then
one sees that $\varphi=x \otimes_b (\alpha x)$ and one notes that $\alpha x\in \{x\}^\bot$.

Now, if $\rk u \geq 2$, one reduces the situation to the one where $\rk u \leq 1$, with exactly
the same line of reasoning as in the proof of Lemma \ref{caractensors}: we leave the details to the reader.
\end{proof}

Next, the Trace Orthogonality Lemma yields the following result:

\begin{lemma}\label{orthotensorherm}
Let $\calV$ be a nilpotent $\K$-linear subspace of $\calH_b$.
Let $x$ be a non-zero vector of $V$. Let $u \in \calV$ and $y \in V$  be such that $x \otimes_b y \in \calV$.
Then, $\Tr_{\F/\K}\bigl(b(y,u(x))\bigr)=0$.
\end{lemma}

\begin{proof}
Set $v:=x \otimes_b y$. By Lemma \ref{tracecor}, we have $\tr(v u)=0$.
Then, with the same line of reasoning as in the proof of Lemma \ref{orthotensorsym}, one checks
that
$$\tr(v u)=b(y,u(x))+b(x,u(y))=b(y,u(x))+b(y,u(x))^\star=\Tr_{\F/\K}\bigl(b(y,u(x))\bigr),$$
which yields the conclusion.
\end{proof}

From there, adapting the proof given in Section \ref{InductiveSection} is easy:
let $\calV$ be a nilpotent $\K$-linear subspace of $\calH_b$.
If $b$ is non-isotropic, then $\calV=\{0\}$ by Lemma \ref{nonisotropiclemmaH}, and we are done.
Assume otherwise, and denote by $\nu>0$ the Witt index of $b$. Choose a vector $x \in V \setminus \{0\}$
such that $b(x,x)=0$. Set $D:=\{\lambda \in \K : \; \Tr_{\F/\K}(\lambda)=0\}$.

Consider the subspace
$$\calU:=\{u \in \calV : \; u(x)=0\}.$$
Denote by $\overline{b}$ the Hermitian form induced by $b$ on $\overline{V}:=\{x\}^\bot/\F x$:
its Witt index equals $\nu-1$.
Let $u \in \calU$. Since $u$ stabilizes $\F x$, it also stabilizes its orthogonal complement $\{x\}^\bot$,
whence it induces a nilpotent endomorphism
$\overline{u}$ of $\overline{V}$. The set
$$\calV \modu x:=\{\overline{u} \mid u \in \calU\}$$
is a nilpotent linear subspace of $\calH_{\overline{b}}$.
By Proposition \ref{caractensors}, the kernel of the surjective $\K$-linear mapping
$$\Phi : u \in \calU \mapsto \overline{u} \in \calV \modu x$$
reads $x \otimes_b L$ for a unique $\K$-linear subspace $L$ of $\{x\}^\bot$
that includes $D$, and we have $\dim_\K \Ker \Phi=\dim_\K L-1$.
On the other hand, we set
$$\calV x:=\{u(x) \mid u \in \calV\},$$
so that $\calU$ is the kernel of the surjective $\K$-linear mapping $u \in \calV \mapsto u(x) \in \calV x$.
Note that $\calV x \cap \F x=\{0\}$ as no element $u \in \calV$ satisfies $u(x)=\mu x$
for some non-zero scalar $\mu\in \F$.

Applying the rank theorem twice, we find
$$\dim_\K \calV=(\dim_\K L-1)+\dim_\K (\calV x)+\dim_\K (\calV \modu x).$$
Yet, by Lemma \ref{orthotensorherm}, the subspaces $L$ and $\calV x$ are $b$-orthogonal.
Better still, since $x$ is $b$-orthogonal to $L$, we find that $\F x \oplus \calV x$ is $b$-orthogonal to $L$, leading to $(\dim_\K (\calV x)+2)+\dim_\K L \leq 2n$, and hence
$$\dim_\K (\calV x)+\dim_\K L \leq 2n-2.$$
Finally, by induction, we have
$$\dim_\K (\calV \modu x) \leq (\nu-1)(2(n-2)-2(\nu-1)-1)=(\nu-1)(2n-2\nu-3),$$
whence
$$\dim_\K \calV \leq (\nu-1)(2n-2\nu-3)+2n-3=\nu(2n-2\nu-1).$$
Moreover, if equality holds, then
$$\dim_\K (\calV \modu x)=(\nu-1)(2(n-2)-2(\nu-1)-1).$$

\subsection{Direct proof of the inequality in Theorem \ref{HermitianTheo}}

Here, we give a direct proof of the inequality statement in  Theorem \ref{HermitianTheo},
by adapting the proof given in Section \ref{DirectMORSection}.

Let $\calV$ be a nilpotent $\K$-linear subspace of $\calH_b$.
We choose a maximal partially complete $b$-singular flag $\calF=(F_0,\dots,F_\nu)$ of $V$, and we set $p:=n-2\nu$.
Just like in Section \ref{geomExamples}, we find a basis $\bfB=(e_1,\dots,e_n)$ of $V$ that is strongly adapted to $\calF$ in the following sense:

\begin{itemize}
\item $(e_1,\dots,e_i)$ is a basis of $F_i$ for all $i \in \lcro 0,\nu\rcro$;
\item $e_{\nu+1},\dots,e_{\nu+p}$ are orthogonal to $e_1,\dots,e_\nu,e_{n-\nu+1},\dots,e_n$;
\item $b(e_i,e_{n-\nu+j})=\delta_{i,j}$ for all $(i,j) \in \lcro 1,\nu\rcro^2$;
\item the subspace $\Vect(e_k)_{n-\nu < k \leq n}$ is totally singular for $b$.
\end{itemize}

In the basis $\bfB$, the matrix of $b$ reads
$$H=\begin{bmatrix}
0_\nu & [0]_{\nu \times p} & I_\nu \\
[0]_{p\times \nu} & P & [0]_{p \times \nu} \\
I_\nu & [0]_{\nu \times p}  & 0_\nu
\end{bmatrix}$$
for some non-isotropic Hermitian matrix $P \in \Math_p(\F)$.
For every $u \in \calH_b$, the matrix of $u$ in $\bfB$ reads
$$M(u)=\begin{bmatrix}
A(u) & (PC_2(u))^\star & D_2(u) \\
C_1(u) & P^{-1} H'(u) & C_2(u) \\
D_1(u) & (PC_1(u))^\star & A(u)^\star
\end{bmatrix}$$
where $A(u) \in \Mat_\nu(\F)$, $D_1(u)$, $D_2(u)$ belong to $\Math_\nu(\F)$, $C_1(u)$ and $C_2(u)$ belong to $\Mat_{p,\nu}(\F)$, and $H'(u) \in \Math_p(\F)$.

As we have seen in Section \ref{HermitianExample}, the elements of $\WH_{b,\calF}$ are exactly the $u \in \calH_b$
for which $C_1(u)=0$, $D_1(u)=0$, $H'(u)=0$, and $A(u)$ is strictly upper-triangular.

To any $u \in \calH_b$, we assign the matrix $I(u) \in \NT_\nu(\F)$ defined as follows:
$$I(u)_{i,j}=\begin{cases}
A(u)_{j,i} & \text{if $i<j$} \\
0 & \text{otherwise.}
\end{cases}$$
Since $\calV$ is nilpotent, it follows that $\calV \cap \WH_{b,\calF}$ is precisely the kernel of the  $\K$-linear mapping
$$\chi : u \in \calV \mapsto (I(u),C_1(u),D_1(u)) \in \NT_\nu(\F) \times \Mat_{p,\nu}(\F) \times \Math_\nu(\F).$$
The rank theorem then yields
$$\dim \calV = \rk \chi+\dim (\calV \cap \WH_{b,\calF}).$$
Finally, let $u \in \WH_{b,\calF}$ and $v \in \calH_b$.
Noting that
$$M(u)=\begin{bmatrix}
A(u) & (PC_2(u))^\star & D_2(u) \\
[0]_{p \times \nu} & 0_p & C_2(u) \\
0_\nu & [0]_{\nu \times p} & A(u)^\star
\end{bmatrix}$$
with $A(u)$ strictly upper-triangular, 
we find
\begin{multline*}
\tr(v  u) =\tr\bigl(I(v)^T A(u)\bigr)+\tr\bigl(C_1(v)^\star P C_2(u)\bigr)\\
+\tr\bigl(C_1(v)C_2(u)^\star P\bigr)+\tr\bigl(D_1(v)D_2(u)\bigr)+\tr\bigl((I(v)^\star)^T A(u)^\star\bigr),
\end{multline*}
which, as $P$ is Hermitian, can be rewritten as follows:
$$\tr(v u)=\Tr_{\F/\K}\bigl(\tr(I(v)^T A(u)\bigr)
+\Tr_{\F/\K}\Bigl(\bigl(\tr(C_1(v)^\star P C_2(u)\bigr)\Bigr)+\tr\bigl(D_1(v)D_2(u)\bigr).$$
Using the Trace Orthogonality Lemma, we deduce that
the range of $\chi$ is $c$-orthogonal to the direct image of $\calV  \cap \WH_{b,\calF}$ under
the $\K$-linear mapping
$$\chi' : u \in \WH_{b,\calF} \mapsto (A(u),C_2(u),D_2(u)) \in \NT_\nu(\F) \times \Mat_{p,\nu}(\F) \times \Math_\nu(\F)$$
for the symmetric $\K$-bilinear form $c$ defined as follows on $\NT_\nu(\F) \times \Mat_{p,\nu}(\F) \times \Math_\nu(\F)$:
$$c\bigl((N,C,D),(N',C',D')\bigr) =
\Tr_{\F/\K}\bigl(\tr(N^TN')\bigr)+\Tr_{\F/\K}\bigl(\tr(C^\star P C')\bigr)
+\tr(DD').$$
Using the fact that $P$ is invertible, one checks that $c$ is non-degenerate, and one deduces that
\begin{align*}
\rk \chi+\dim_\K \bigl(\chi'(\calV \cap \WH_{b,\calF})\bigr)
& \leq \dim_\K \bigl(\NT_\nu(\F) \times \Mat_{p,\nu}(\F) \times \Math_\nu(\F)\bigr) \\
& \leq 2\dbinom{\nu}{2}+2p\nu+\nu^2=\nu(2n-2\nu-1).
\end{align*}
Since $\chi'$ is obviously injective, we conclude that
$$\dim_\K \calV =\rk \chi+\dim_\K \bigl(\chi'(\calV \cap \WH_{b,\calF})\bigr)\leq \nu(2n-2\nu-1).$$
As in Section \ref{DirectMORSection}, the following conclusion can be drawn from the above proof and from
Lemma \ref{stableflaghermitian}:

\begin{lemma}\label{doubleortholemmaHermitian}
Let $\calV$ be a nilpotent $\K$-linear subspace of $\calH_b$ with dimension $\nu(2n-2\nu-1)$.
Let $v \in \calH_b$ be nilpotent and such that $\forall u \in \calV, \; \tr(uv)=0$.
Then, $v \in \calV$.
\end{lemma}

\subsection{The case of equality in Theorem \ref{HermitianTheo}}

The line of reasoning here is similar to the one of the proof of Theorem \ref{EqualitySymSymtheo}.
First of all, we note that the square of a $b$-Hermitian endomorphism is still a $b$-Hermitian endomorphism:
indeed, we have noted that $u \in \End_\F(V)$ is $b$-Hermitian if and only if
$$\forall (x,y)\in V^2, \; b(x,u(y))=b(u(x),y),$$
and it is obvious that if $u$ satisfies this property then so does $u^2$.

Next, we take a nilpotent $\K$-linear subspace $\calV$ of $\calH_b$ with dimension $\nu(2n-2\nu-1)$. We assume that $|\F|>4$, so that $|\K|>2$.

Let $u \in \calV$.
Let $v \in \calV$. Hence, $v$ is nilpotent, and $u+\lambda v$ is nilpotent for all $\lambda \in \K$. Since $|\K|>2$, the Trace Orthogonality Lemma
yields $\tr(u^2 v)=0$. Hence, Lemma \ref{doubleortholemmaHermitian} shows that $u^2$ belongs to $\calV$.

We have just shown that $\calV$ is stable under squares, and we deduce, as in the proof of Proposition \ref{EqualitySymSymprop}, that it is stable under the Jordan product. Using Jacobson's triangularization theorem, we arrive at the following partial result:

\begin{prop}\label{EqualityHermHermprop}
Let $b$ be a non-degenerate Hermitian form on an $n$-dimensional vector space $V$ over $\F$. Denote by $\nu$ the Witt index of $b$, and let $\calV$ be a nilpotent $\K$-linear subspace of $\calH_b$ with dimension $\nu(2n-2\nu-1)$. Assume finally that $|\F|>4$. Then, $\calV$ is triangularizable.
\end{prop}

From there, the line of reasoning of Section \ref{MaxDimSection} applies effortlessly to yield the statement on the case of equality in Theorem \ref{HermitianTheo}.

\end{document}